\documentclass{amsart}
\usepackage{latexsym,amsmath,amscd, amssymb}
\allowdisplaybreaks
\newtheorem{theorem}{Theorem}[subsection]
\newtheorem{corollary}[theorem]{Corollary}
\newtheorem{lemma}[theorem]{Lemma}
\newtheorem{proposition}[theorem]{Proposition}
\theoremstyle{definition}
\newtheorem{definition}[theorem]{Definition}
\theoremstyle{emp}

\theoremstyle{notat}

\theoremstyle{claim}
\newtheorem{claim}[theorem]{Claim}
\theoremstyle{remark}
\newtheorem{remark}[theorem]{Remark}
\theoremstyle{example}
\newtheorem{example}[theorem]{Example}
\numberwithin{equation}{subsection}
\newcommand{\isomto}{\overset{\sim}{\rightarrow}}

\title{ C\MakeLowercase{yclotomic} $p$-\MakeLowercase{adic} M\MakeLowercase{ulti}-Z\MakeLowercase{eta} V\MakeLowercase{alues} \MakeLowercase{in} D\MakeLowercase{epth} T\MakeLowercase{wo} }

\author{ S\MakeLowercase{inan} \"{U}\MakeLowercase{nver}}
\address{Ko\c{c} University, Mathematics Department. Rumelifeneri Yolu, 34450, Istanbul, Turkey  }

\email{sunver@ku.edu.tr  }

\subjclass{ }


\begin{document}
\maketitle
\noindent

\begin{abstract} 
In this paper  we compute the values of the $p$-adic multiple polylogarithms of depth two at roots of unity. Our method is to solve  the fundamental differential equation satisfied by the crystalline frobenius morphism using rigid analytic methods. The main result could be thought of as a computation in the $p$-adic theory of higher cyclotomy. We expect the result to be useful in proving non-vanishing  results since it gives quite explicit formulas. 
\end{abstract}

\section{introduction}

Let $\mathbb{G}_{m}:={\rm Spec} \, \mathbb{Q}[z,z^{-1}]$ denote the multiplicative group and  $\mu_{M},$ the kernel of multiplication by $M$ on $\mathbb{G}_{m},$ for  $M \geq1.$ Let $Y_{M}:=\mathbb{G}_{m} \setminus \mu_{M},$ and $Y_{M} '$ denote the base change of $Y_{M}$ to $L:=\mathbb{Q}(\zeta),$ where $\zeta$ is a primitive $M$-th root of unity.  Choosing appropriate (tangential) basepoints Deligne and Goncharov define the unipotent motivic fundamental group $\pi_{1} ^{mot} (Y_{M} ' , \cdot)$ of $Y_{M} '$  whose ring of functions is  an ind-object in the tannakian category  of mixed Tate motives over $\mathcal{O}_{L}[M^{-1}]$ \cite[\textsection 5]{DG}.

The comparison between the Betti and de Rham realization of this fundamental group is completely described by the cyclotomic versions of multi-zeta values \cite[Proposition 5.17]{DG}.  
Namely, fix an  imbedding of $L$ in $\mathbb{C},$ and identify $L$ with its image. The Lie algebra of the de Rham fundamental group of $Y_{M} '$ over $\mathbb{C}$ is the free pro-nilpotent Lie algebra with  generators $\{ e_{i}\}_{0 \leq i \leq M},$ where $e_{0}$  (resp. $e_{i}$) correspond to taking residues at 0 (resp. $\zeta^{i}$) (cf. \textsection \ref{connection on the fundamental torsor} below). Hence the $\mathbb{C}$-valued points of the de Rham fundamental group is the set of group-like elements in the  (non-commutative) formal power series ring $\mathbb{C} \langle \langle e_{0}, \cdots, e_{M} \rangle \rangle $ (loc. cit.).  The image of the  Betti path from the tangential basepoint 1 at 0 to the tangential basepoint -1 at 1 under the de Rham-Betti comparison isomorphism then gives a group-like element $_{1}\gamma_{0}$ in $\mathbb{C} \langle \langle e_{0}, \cdots, e_{M} \rangle \rangle .$ Then by \cite[Proposition 5.17]{DG}, the coefficient of $e_{0} ^{s_{m}-1}e_{i_{m}} \cdots e_{0} ^{s_{1}-1}e_{i_{1}}$  in $_{1}\gamma_{0},$ where $M \geq i_{m}, \cdots,i_{1}  \geq 1$ and $s_{m}>1,$  is 
$$
(-1)^{m} \sum _{n_{m}> \cdots >n_{1}>0}\frac{\zeta ^{ i_m(n_{m-1}-n_{m} ) + \cdots + (-i_{1}n_{1})  }  }{n_{m} ^{s_{m}} \cdots n_{1} ^{s_{1}}}.
$$
The study of these numbers is the Hodge-theoretic analog of higher cyclotomy \cite{gon}. 

The main result below is the crystalline analog of the above for $m \leq 2.$ We describe this in more detail.  Letting $p$ be a prime which does not divide $M,$  $\pi_{1} ^{mot} (Y_{M} ',\cdot)$  has good reduction modulo $p,$ and hence one would expect a crystalline realization of this motive at $p.$ This is completely described by the frobenius action on  the de Rham fundamental group $\pi_{1,dR}(X_{M} ', \cdot),$ where $X_{M} '$ is the base change of  $Y_{M}'$  to $\mathbb{Q}_{p} (\zeta).$  Let $g_{i}$ denote the image of the canonical de Rham path from the tangential basepoint 1 at 0 to the tangential basepoint 1 at $\zeta ^{i/p}$ (c.f. \textsection \ref{Cyclotomic p-adic multi-zeta values}).  As above $g_{i}$ is naturally in $\mathbb{Q}_{p}(\zeta) \langle \langle e_{0}, \cdots, e_{M} \rangle \rangle .$ The main result of the paper,  Theorem \ref{Theorem main} below,  gives an explicit formula for the coefficient of $e_{j}e_{0} ^{s-1} e_{k} e_{0} ^{t-1}$ in $g_{i},$ in terms of iterated sums, exactly as above.  Since $g_{i}$ is group-like, this also determines  the coefficients of terms of the form $e_{0} ^{r-1}e_{j}e_{0} ^{s-1} e_{k} e_{0} ^{t-1}.$  This might be thought of as the $p$-adic theory of higher cyclotomy in depth two.

Next we describe the contents of the paper.  In \textsection 2, we review the de Rham and crystalline fundamental groups of a curve in a way which will be suitable for our purposes. In particular, using the horizontality of the frobenius with respect to the canonical connection we arrive at the fundamental differential equation (\ref{diff eq for frobenius}). At the end of this section, we fix the notation for what follows. 
In \textsection 3, we obtain a certain relation between the coefficients of the power series expansions of rigid analytic functions on $\mathcal{U}_{M},$  which is essential for the computations (Corollary \ref{cor of main rigid}). In \textsection 4, we  compute of the polylogarithmic part which is fairly straightforward. Next there is a section on the type of iterated sums that appear in the computations. These functions will appear as coefficients of the power series expansions above and will satisfy the hypotheses of Corollary \ref{cor of main rigid}, so the inductive process will continue. In \textsection 6, we will proceed with the computation, and finish with the main result in Theorem \ref{Theorem main}.

\section{The fundamental differential equation}

Fix a prime $p,$ which does not divide $M.$ Let $X_{M}$ and $X_{M} '$ denote the base change of $Y_{M}$ to $\mathbb{Q}_{p}$ and $\mathbb{Q}_{p} (\zeta)$. Let $A_{M}$ and $A_{M} '$ denote the rings of  regular functions on $X_{M}$ and $X_{M} '$ respectively.   Finally, let $D_{M}  := \overline{X}_{M} ' \setminus X_{M} ' .$

\subsection{The de Rham fundamental group of $X_{M}' $} We review the theory of the de Rham fundamental group \cite[10.24-10.53, \textsection 12]{De}, \cite[\textsection 4, \textsection 5]{un} in the case of  $X_{M}'.$ 

\subsubsection{The fundamental torsor}\label{fundamental torsor} 
 
Let $K$ be any field of characteristic 0 and $X/K$ be a smooth and geometrically connected curve and ${\rm Mic}_{uni}(X/K)$ denote the category of vector bundles with integrable connection which are { \it unipotent}, i.e. that have a finite separated and exhaustive filtration by sub-bundles with connection such that the non-zero graded pieces are the trivial line bundle with connection. This category naturally forms a tensor category over $K$ in the sense of \cite[\textsection 5.2]{De}, \cite{De-Tann}.

Let $S/K$ be a scheme over $K$ and let ${\rm Vec}_{S}$ denote the category of locally free sheaves of finite rank on $S$ and   $\omega: {\rm Mic}_{uni}(X/K) \to {Vec}_{S}$ be a fiber functor \cite[\textsection 5.9]{De}. Then to $\omega$ there is associated a $K$-groupoid acting over $S$ \cite[\textsection 1.6]{De-Tann} called the {\it fundamental groupoid} of $X$ at $\omega$ and denoted by $\mathcal{P}_{dR}(X,\omega) .$  The fundamental groupoid is faithfully flat and affine over $S\times _{K} S$ and represents the functor on the category of $S \times _{K} S$-schemes whose $T$-valued points for any  $\pi: T \to S \times _{K} S $ is the set of $\otimes$-isomorphisms from $ \pi ^* p_{2} ^{*} \omega$ to $\pi ^* p_{1} ^{*} \omega ,$ where $p_{1} , \, p_{2} : S\times _{K} S \to S$ are the projections \cite[\textsection 1.11, Th\'{e}or\`{e}me 1.12]{De-Tann}.  
  
 Taking the cartesian product of   $\mathcal{P}_{dR}(X, \omega)  \to S \times _{K} S $ with the diagonal $\Delta : S  \to S \times_{K} S$ gives  $\pi_{1,dR} (X, \omega ),$  the {\it fundamental group}  of  $X $ at the fiber functor $\omega . $ 
 
 Let $x \in S(K)$ then attaching $\mathcal{F}(x),$ the fiber of $\mathcal{F}$ at $x,$ to $\mathcal{F} \in {\rm Vec}_{S}$ gives rise to a fiber functor 
 $$
 \omega _{x}: {\rm Mic}_{uni} (X/K) \to {\rm Vec}_{K} .
 $$  
 
 Pulling back $\mathcal{P}_{dR}(X, \omega) \to S \times _{K} S$ via the inclusion $S \to S \times _{K}S$ that sends $s$ to $(s,x)$ we obtain a torsor  $\mathcal{T}_{dR}(X,\omega)_{x} $ on $S$ under the group scheme $\pi_{1,dR}(X,\omega_{x}).$

\subsubsection{The de Rham fiber functor on $X_{M}$}\label{de Rham fiber functor} From now on we assume that the smooth projective model $\overline{X}$ of $X$ is isomorphic to  $\mathbb{P}^{1}.$ In this case,  there is a canonical fiber functor \cite[\textsection 12]{De}:
$$
\omega(dR): {\rm Mic}_{uni}(X/K) \to {\rm Vec}_{K}
$$
defined as follows. 

For any $(E,\nabla) \in {\rm Mic}_{uni}(X/K)$ let $(E_{can},\nabla)$ denote the unique vector bundle with connection on $\overline{X}$ that has  logarithmic singularites with nilpotent residues at $\overline{X}\setminus X.$  The pair $(E_{can},\nabla)$ is called the canonical extension of the unipotent vector bundle with connection $(E,\nabla).$ Since $H^{1}(\overline{X},\mathcal{O})=0,$ the bundle $E_{can}$ is trivial \cite[Proposition 12.3]{De} and the functor $\omega_{dR}$ defined as 
$$
\omega_{dR} (E,\nabla):= \Gamma(\overline{X},E_{can})
$$
is a fiber functor \cite[\textsection 12.4]{De}. For a subscheme $Y$ of $\overline{X}$ let 
$$
\omega(Y): {\rm Mic}_{uni}(X/K) \to {\rm Vec}_{Y}
$$ 
denote the fiber functor that sends $(E,\nabla)$ to $E_{can}|_{Y}.$  There are canonical isomorphisms
\begin{align}\label{isom of fiber}
\omega_{dR} \otimes _{K} \mathcal{O}_{Y}\cong \omega(Y)
\end{align}
of fiber functors.

Let $\mathcal{P}_{dR}:= \mathcal{P}_{dR}(X,\omega(X)),$ $\mathcal{T}_{dR,x}:=\mathcal{T}_{dR}(X, \omega(X) )_{x},$  $\overline{\mathcal{P}}_{dR}:= \mathcal{P}_{dR}(X,\omega(\overline{X}))$ and $\overline{\mathcal{T}}_{dR,x}:=\mathcal{T}_{dR}(X, \omega(\overline{X}) )_{x}.$ Finally let $\mathcal{T}_{dR}$ and $\overline{\mathcal{T}}_{dR}$ denote the torsors $\mathcal{T}_{dR,x}$ and $\overline{\mathcal{T}}_{dR,x}$ after the identification (\ref{isom of fiber}) of $\omega_{dR}$ with $\omega(x).$ Thus they are torsors under $\pi_{1,dR}(X):=\pi_{1,dR}(X,\omega_{dR})$  and depend only on $X.$  
 
\subsubsection{Connection on the fundamental torsor}\label{connection on the fundamental torsor}  Let $\Delta_{X}$ denote the diagonal in $X \times _{K}X  $ and $\Delta_{X} ^{(1)}$ denote its first infinitesimal neighborhood. By the definition of $\mathcal{P}_{dR},$ the sections of its restriction to $\Delta_{X} ^{(1)}$ are $\otimes $-isomorphisms from $p _{2} ^{(1)*} \omega(X)$ to $p_{1} ^{(1)*} \omega(X)$ where $p_{i} ^{(1)} : \Delta_{X} ^{(1)} \to X $ are the two projections. If $(E,\nabla) \in {\rm Mic}_{uni}(X/K),$ then  $p _{i} ^{(1)*} \omega(X)(E,\nabla)=p _{i} ^{(1)*}(E) $ and the connection $\nabla$ induces an isomorphisms from  $p _{2} ^{(1)*}(E)$ to $p _{1} ^{(1)*}(E)$ reducing to the identity on the diagonal. This in turn induces an isomorphism between the above fiber functors, and hence a section of $\mathcal{P}_{dR} |_{\Delta_{X} ^{(1)}}$ over  $\Delta_{X} ^{(1)}$ which is the identity section when restricted to $\Delta_{X} .$

Therefore there is a   canonical section of the restriction of  $\mathcal{P}_{dR}$  to $\Delta_{X} ^{(1)}$ which is the identity section on $\Delta_{X}.$ This gives  a connection on the $\pi_{1,dR}(X)$-torsor $\mathcal{T}_{dR}.$ Note that because of the canonical isomorphisms (\ref{isom of fiber})  $\mathcal{T}_{dR}\cong \pi_{1,dR}(X) \times_{K} X $ and $\overline{\mathcal{T}}_{dR}\cong \pi_{1,dR}(X) \times _{K} \overline{X}.$ A connection on $\mathcal{T}_{dR}$ is  nothing other than a morphism $\Delta_{X} ^{(1)} \to \pi_{1,dR}(X)$ whose restriction to $\Delta_{X}$ is the constant map to the identity element of $\pi_{1,dR}(X).$ This in turn is equivalent to giving a section of ${\rm Lie}\, \pi_{1,dR}(X) \otimes_{K} \Gamma(X, \Omega ^{1} _{X/K}).$  By \cite[\textsection 12.12]{De},  the connection on $\mathcal{T}_{dR}$ is the one that corresponds to the canonical section $\alpha$ of 
$H_{dR} ^{1}(X) \check{ } \otimes _{K}H^{1} _{dR}(X) \subseteq {\rm Lie}\, \pi_{1,dR}(X) \otimes_{K} \Gamma(X, \Omega ^{1} _{X/K}).$

From now on we let $X=X_{M} '$ and $K=\mathbb{Q}_{p} (\zeta).$ The image of $\alpha$ in ${\rm Lie}\, \pi_{1,dR}(X) \otimes_{K} \Gamma(X, \Omega ^{1} _{X/K})$ under the canonical maps above can be described as follows. For any $x \in \overline{X} \setminus X$ and  $(E,\nabla),$ we have the residue endomorphism 
$$
{\rm res}_{x}: E_{can}(x) \to E_{can}(x),
$$
induced by the map that sends the local section $u$ of $E_{can}$ near $x,$ to $(\nabla(u), t \frac{\partial}{\partial t}),$ where $t$ is a uniformizer at $x.$ The residue endomorphism is independent of the choice of a uniformizer and satisfies, 
$$
{\rm res}_{x}((E_{1},\nabla_{1}) \otimes (E_{2},\nabla_{2}))=1 \otimes {\rm res}_{x}(E_{2},\nabla_{2})+{\rm res}_{x}(E_{1},\nabla_{1})\otimes 1.
$$
Hence ${\rm res}_{x} \in {\rm Lie} \, \pi_{1,dR}(X, \omega(x)).$

Under the identification (\ref{isom of fiber}),  for $1 \leq i \leq M,$ we let $e_{i} \in {\rm Lie}\, \pi_{1,dR}(X_{M} ') $ correspond to ${\rm res}_{\zeta^{i}}$ and $e_{0}$ to ${\rm res}_{0}.$ If we also put $\omega_{0}:=d log z$ and $\omega_{i}:= d log (z-\zeta ^{i}),$ for $1 \leq i \leq M,$ then the section of ${\rm Lie} \, \pi_{1,dR}(X_{M} ')\hat{ \otimes} _{K} \Gamma(X,\Omega^{1} _{X_{M} '/K})   $ that corresponds to the connection on $\mathcal{T}_{dR}$ is 
$
\sum _{0 \leq i \leq M}e_{i} \omega_{i}.
$

The de Rham fundamental group of $X_{M} '$ has a simple description. For any $K$-algebra $A,$  denote the 
associative (non-commutative) algebra of  formal power series in $\{ e_{i} | 0 \leq i \leq m \}$ over $A$ by $A \langle \langle e_{0}, \cdots , e_{M}  \rangle \rangle $ and let 
$$
\mathcal{U}_{dR}(A):= A \langle \langle e_{0}, \cdots , e_{M}  \rangle \rangle.
$$  
Then the universal enveloping algebra of $\pi_{1,dR}(X_{M} ')$ is 
$
\mathcal{U}_{dR}(X_{M} ')(K).$The co-product of the Hopf algebra structure on $\mathcal{U}_{dR}(A)$ is induced by the fact that $e_{i}$ are primitive elements:  $\Delta(e_{i}) =1 \otimes e_{i} + e_{i} \otimes 1,$ for  $1 \leq i \leq M.$ The $A$-valued points of $\pi_{1,dR}(X_{M} ')$ then correspond to the group-like elements in $\mathcal{U}_{dR}(A),$ i.e. elements $g$ satisfying $\Delta(g)=g \hat{\otimes} g$ and  with constant term equal to 1.  For any $g$ let $\underline{g}$ denote the image of $g$ under the Hopf algebra automorphism of $\mathcal{U}_{dR}(A)$ that sends $e_{i}$ to $p^{-1}e_{i},$ for all $i.$ 
 
The canonical connection on $\mathcal{T}_{dR}=\pi_{1,dR}(X_{M} ')\times X_{M} '$ can be described as follows. A section of $\mathcal{T}_{dR}$ over  $X_{M} '$ is given by a group-like element in $\alpha(z) \in \mathcal{U}_{dR}(A_{M} '),$ where $z$ denotes the parameter on ${\rm Spec}\, A_{M} '=X_{M}' \subseteq \mathbb{A}^{1} _{K}.$  Let 
$$
d: \mathcal{U}_{dR} (A_{M} ') \to \mathcal{U}_{dR}(A_{M} ') \hat{\otimes} _{A_{M} '} \Omega^{1} _{A_{M} ' /K}  
$$
denote the continuous differential extending the canonical differential $A_{M} ' \to \Omega^{1} _{A_{M} ' /K}$ such that $d(e_{i})=0,$ for $0 \leq i \leq M.$ In other words, applying $d$ to an element $\alpha(z)$ amounts to applying $d$ to each coefficient of $\alpha(z).$ 

With this notation, the image of $\alpha(z) $  in ${\rm Lie}\, \pi _{1,dR}(X_{M} ') \hat{\otimes} \Omega^{1} _{A_{M} '/K}$ under the canonical  connection $\nabla$ on $\mathcal{T}_{dR}$ is:

$$
\nabla(\alpha(z))=  \alpha(z) ^{-1} d \alpha(z)- \alpha(z)^{-1}(\sum _{0 \leq i \leq M}e_{i}\omega_{i})\alpha(z).
 $$

\subsection{Crystalline fundamental group of $X_{M} '$}  
We review the theory of the crystalline fundamental group as described in \cite[\textsection 11]{De} and \cite[\textsection 2.4]{un}. The comparison theorem between the crystalline and de Rham fundamental groups will give us the frobenius map which will be central to what follows. 

\subsubsection{The de Rham-crystalline comparison} 
Let  $k$ be a perfect field of characteristic $p,$ with $W$ the ring of Witt vectors and $K$ its field of fractions. For a smooth variety $Y/k,$ we have  ${\rm Isoc}_{uni} ^{\dagger} (Y/W),$ the category of unipotent overconvergent isocrystals on $Y/W$ \cite[\textsection 2.4.1]{un}, whose fundamental group at a fiber functor $\omega$ is the crystalline fundamental group, $\pi_{1,crys} ^{\dagger}(Y,\omega),$ of $Y.$  Now suppose that  $Y$ has a smooth compactification $\overline{Y}/k$ such that $D:=\overline{Y} \setminus Y$ is a simple normal crossings divisor in $\overline{Y},$ and let $\overline{Y}_{\log}$ denote the canonical log structure on $\overline{Y}$ associated to the divisor $D.$   Shiho's theorem \cite{shi} implies that  the restriction functor  
\begin{align}\label{log crys}
{\rm Isoc}_{uni} ^{c} (\overline{Y}_{log}/W) \to {\rm Isoc}_{uni} ^{\dagger} (Y/W)
,
\end{align} from the category of unipotent convergent log isocrystals on $\overline{Y}_{log}$ to ${\rm Isoc}_{uni} ^{\dagger} (Y/W),$ is an equivalence of categories \cite[Lemma 2]{un}. 
  
 This is in complete analogy with the situation over the field $K$ of characteristic 0. If $X/K$ is a smooth variety with a smooth compactification $\overline{X}/K $ and with  simple normal crossings divisor $E:=\overline{X}\setminus X$ in $\overline{X}$ then the restriction 
 \begin{align}\label{log de rham}
 {\rm Mic}_{uni} (\overline{X}_{\log}/K) \to {\rm Mic}_{uni} (X/K)
 \end{align}
 gives an equivalence of categories \cite[II.5.2]{De-eq}. 
 
 The de Rham-crystalline comparison can be described as follows.  Suppose that $\overline{\mathcal{Z}}/W$ is a smooth, projective scheme with geometrically connected fibers and  with  $\mathcal{F} \subseteq \overline{\mathcal{Z}}$  a relative simple normal crossings divisor. Let $\mathcal{Z}:= \overline{\mathcal{Z}}\setminus \mathcal{F},$ and let $(\overline{X}, X, E)$ and $(\overline{Y},Y,D)$ denote the corresponding data over the generic and special fibers respectively.   
 The canonical functor
 $$
  {\rm Mic}_{uni} (\overline{X}_{\log}/K) \to {\rm Isoc}_{uni} ^{c} (\overline{Y}_{log}/W)
 $$ 
is an equivalence which, when combined with (\ref{log de rham}) and (\ref{log crys}) gives the equivalence 
\begin{align}
  {\rm Mic}_{uni} (X/K) \to {\rm Isoc}_{uni} ^{\dagger} (Y/W).
 \end{align}

Choosing a (tangential) basepoint $z$ on $\mathcal{Z},$ we get an isomorphism 
 $$
 \pi_{1,crys} ^{\dagger}(Y,y) \isomto \pi_{1,dR}(X,x) ,
 $$ where $x$ and $y$ are the generic and special fibers of $z.$ 
 
Let $\sigma: W \to W$ denote the lifting of the $p$-power  frobenius map on $k,$ and let $\mathcal{Z}^{(p)},$ denote the base change of $\mathcal{Z}/W$ via $\sigma$ and  $X^{(p)}, \, Y^{(p)}$ etc. the corresponding fibers of $\mathcal{Z}^{(p)}.$ The relative frobenius morphism induces a $\otimes$-functor $F^{*}: {\rm Isoc} ^{\dagger} _{uni} (Y^{(p)}/W) \to {\rm Isoc} ^{\dagger} _{uni} (Y/W) ,$ and hence a map $F_{*} :  \pi_{1,crys} ^{\dagger} (Y,y) \to \pi_{1,crys} ^{\dagger} (Y^{(p)},y^{(p)}).$  This, together with the above isomorphism, gives a morphism\begin{align}\label{frob on de rham}
F_{*}: \pi_{1,dR}(X,x) \to \pi_{1,dR}(X^{(p)},x^{(p)}).
\end{align}
Similarly, for a pair of (tangential) basepoints $z_{1}$ and $z_{2}$ we obtain a morphism
\begin{align}
F_{*}:  \,_{x_{2}}\mathcal{P}_{dR}(X)_{x_{1}} \to  \,_{x_{2} ^{(p)}}\mathcal{P}_{dR}(X^{(p)})_{x_{1} ^{(p)}}.
\end{align}

\subsubsection{Tangential basepoints in  the crystalline case} Tangential basepoints in dimension 1 are explained in detail in \cite[\textsection 15]{De} and \cite[\textsection 3]{un}. 

Let $\mathcal{Z}/W$ be as above with relative dimension 1, for simplicity, and let $z \in (\overline{\mathcal{Z}} \setminus \mathcal{Z})(W)$ with fibers $x$ and $y.$    Let $T_{z} ^{\times}(\overline{ \mathcal{Z}})/W$ denote the tangent space of $\overline{ \mathcal{Z}}$ at $z$ with the zero section removed. It is (non-canonically) isomorphic to $\mathbb{G}_{m}/W.$ Fix $w \in T_{z} ^{\times}(\overline{ \mathcal{Z}})(W),$ with fibers $v \in T_{y} ^{\times}(\overline{ Y})(k)$ and $u \in T_{x} ^{\times}(\overline{ X})(K).$ The {\it crystalline tangential basepoint at} $u$ is a fiber functor 
$$
\omega_{u} :  {\rm Isoc} ^{\dagger} _{uni} (Y/W) \to {\rm Vec}_{K}. 
$$
Corresponding to the lifting $\mathcal{Z}, z $ and $w$ and the identification of ${\rm Isoc}_{uni} ^{\dagger} (Y/W)$ with ${\rm Mic}_{uni} (X/K)$ described above this fiber functor corresponds to the fiber functor 
$$
\omega_{v}: {\rm Mic}_{uni}(X/K) \to {\rm Vec}_{K}
$$
which associates to $(E,\nabla)$ the fiber $E_{can}(x)$ of its canonical extension at $x.$ The effect of choosing different liftings and the frobenius action of the tangential basepoint are explained in detail in \cite[\textsection 3]{un}.

\subsubsection{Cyclotomic p-adic multi-zeta values}\label{Cyclotomic p-adic multi-zeta values} Let $t_{0}$ denote the tangent vector 1 at 0 and $t_{i}$ denote the tangent vector 1 at $\zeta^{i},$ for $1 \leq i \leq M.$ Also for $1\leq i \leq M,$ let $\overline{i}$ denote  the unique integer  such that $1 \leq \overline{i} \leq M$ and $M| ( \overline{i} -pi).$ Similarly, let $\underline{i}$ denote the unique integer such that $1 \leq \underline{i} \leq M$ and $M | (i - p \underline{i})$ and let $\underline{0}=\overline{0}=0.$  By (\ref{isom of fiber}), for (tangential) basepoints $x_{i}$ on $X_{M} ',$ there are canonical isomorphisms between $\omega_{x_{i}}.$ This gives a canonical element $_{x_{2}}\gamma  _{x_{1}}$ of $_{x_{2}}\mathcal{P}_{dR}(X_{M} ')_{x_{1}},$ which we call the {\it canonical de Rham path} from $x_{1} $ to $x_{2}.$ 
 
For any $1 \leq i \leq M,$ we have elements $_{t_{0}}\gamma_{t_{i}}\cdot F_{*}(\, _{t_{\underline{i}}}\gamma_{t_{0}})  \in \pi_{1,dR}(X_{M} ' , t_{0})(K),$ with $K=\mathbb{Q}_{p}(\zeta).$  Identifying $\omega_{t_{0}}$ with $\omega_{dR}$ using (\ref{isom of fiber}), we obtain elements 
$$
g_{i} \in \pi_{1,dR}(X_{M} ' )(K) \subseteq K\langle \langle e_{0}, \cdots , e_{M} \rangle \rangle,
$$  
for $1 \leq i \leq M.$ Let $g:=g_{M}.$ We denote the coefficient of the monomial  $e_{i_{1}}\cdots e_{i_{n}} $ in $g$ by $g[e_{i_{1}} \cdots e_{i_{n}}]$ and call it a {\it cyclotomic p-adic multi-zeta value}.

The $p$-adic cyclotomic multi-zeta values completely determine the frobenius action on $\pi_{1,dR}(X_{M} ', t_{0})\isomto \pi_{1,dR}(X_{M} ')$ as follows.  

First note that
\begin{align}
 F_{*}(e_{0})=pe_{0},  \; F_{*}(e_{i})=p g_{\overline{i}} ^{-1}e_{\overline{i}} g_{\overline{i}}.
 \end{align}
On the other hand, all the $g_{i}$ are determined by $g$ through functoriality.   
Let $\alpha_{i}$ denote the automorphism of $X_{M} '$ given by $\alpha_{i}(z)=\zeta ^{i} z. $   Then 
$\alpha_{i *}(e_{0})=e_{0}$ and $\alpha_{i*}(e_{j})=e_{i+j},$ where $i+j$ is between $1$ and $M$ computed modulo $M.$ On the special fiber we have $F \circ \alpha_{\underline{i}}=\alpha_{i} \circ F.$ By the functoriality of frobenius we have 
\begin{align}\label{alpha eq}
\alpha_{i*}(g_{j})=g_{i+j}.
\end{align}

\subsubsection{The differential equation satisfied by the frobenius}
Let us first recall the explicit description of the frobenius on ${\rm Mic}_{uni}(\overline{X}_{M,\log} '/K),$ explained in detail in \cite[\textsection 2.4.2]{un}. Let $\overline{\mathcal{P}}_{M}/W $ denote the formal scheme  which  is the completion of $\overline{X}' _{M} $ along the closed fiber and let $\mathcal{D}_{M}$ denote the divisor obtained by completing $D_{M}.$ Let $ \{ \overline{\mathcal{P}}_{i} \}_{1 \leq i\leq n}$  be an open cover of $\overline{\mathcal{P}}_{M},$ and $\mathcal{F}_{i}: \overline{\mathcal{P}}_{i} \to \overline{\mathcal{P}}_{i}$ be a lifting of the frobenius such that $\mathcal{F} _{i} ^{*} (\mathcal{D}_{M} \cap \overline{P}_{i})=p \cdot (\mathcal{D}_{M} \cap \overline{P}_{i}).$ For a formal scheme $\mathcal{P}/W,$ let $\mathcal{P}_{K}$ denote the associated rigid analytic space over $K.$ Now given $(E,\nabla)$ in ${\rm Mic}_{uni} (\overline{X}' _{M,\log}),$ its   pull-back via the frobenius is defined as the vector bundle with connection whose restriction to $\overline{P}_{iK}$ is given by $\mathcal{F}_{i,K} ^{*} (E,\nabla) | _{\overline{\mathcal{P}}_{i,K}  }.$ The isomorphisms between the different pull-backs are given by using the fact that the connections converge within a $p$-adic disk of radius one and that the different liftings of the frobenius lie in the same disk \cite[\textsection 2.4.2]{un}.

Let $\overline{\mathcal{P}}$ denote the completion of $X_{M} ' \cup \{ 0, \infty\}$ along the closed fiber and let  $\mathcal{F}(z)=z^{p}.$ Then $\mathcal{F}: \overline{\mathcal{P}} \to \overline{\mathcal{P}}$ is a lifting of frobenius that satisfies $\mathcal{F} ^{*} (0)=p\cdot (0) $ and $\mathcal{F}^{*} (\infty)=p \cdot (\infty).$   We identify the $\pi_{1,dR}(X_{M} ', t_{0} )$-torsor of paths that start at $t_{0}$  (\textsection \ref{fundamental torsor}) with $\overline{\mathcal{T}}_{dR}$ (\textsection \ref{connection on the fundamental torsor}) by using the identification of $\omega({t_{0}})$ and $\omega_{dR}$ (\textsection \ref{de Rham fiber functor}). The principal part of $\mathcal{F}$ sends $t_{0}$ to itself \cite[\textsection 3.2.(ii)]{un}. Then by the description of the frobenius map above, we obtain the  following commutative diagram:

\[
\begin{CD}
\overline{\mathcal{T}}_{dR} | _{\mathcal{U}_{M} }  @>{\mathcal{F}_{*}}>> \mathcal{F}^{*} \overline{\mathcal{T}}_{dR}|_{\mathcal{U}_{M}} \\
@VV{\nabla}V @VV{\mathcal{F}^* \nabla}V\\
{\rm Lie} \, \pi_{1}(X_{M} ')\hat{\otimes} \Omega^{1} _{\mathcal{U}_{M}}(log(0)) @>{\rm Lie} F_{*}>> {\rm Lie} \, \pi_{1}(X_{M} ')\hat{\otimes} \Omega^{1} _{\mathcal{U}_{M}}(log(0)),
\end{CD}
\]
where $\mathcal{U}_{M} := \overline{\mathcal{P}}_{K}.$ Let $\mathcal{A}_{M}$ denote the ring of rigid analytic functions on $\mathcal{U}_{M}.$ 
Applying frobenius to the section $\gamma_{t_{0}}$ of $\overline{\mathcal{T}}_{dR}| _{\mathcal{U}_{M}} $ and denoting $\mathcal{F}_{*}(\gamma_{t_{0}}) \in \mathcal{U}_{dR}(\mathcal{A}_{M}) $ by $g_{\mathcal{F}}(z),$ we obtain the differential equation
\begin{align}\label{diff eq for frobenius}
-\sum _{0 \leq i \leq M}F_{*}(e_{i}) \omega_{i}=
g_{\mathcal{F}}^{-1} dg_{\mathcal{F}}-g_{\mathcal{F}}^{-1} (\sum_{0 \leq i \leq M} e_{i} \mathcal{F}^{*} \omega_{i}   )g_{\mathcal{F}}.
\end{align}

Putting $g_{0}=1,$ and $\underline{0}=\overline{0}=0,$ we can rewrite this as: 
\begin{align}\label{rewritingtheequation}
dg_{\mathcal{F}}=
 (\sum_{0 \leq i \leq M} e_{i} \mathcal{F}^{*} \omega_{i}   )g_{\mathcal{F}}-g_{\mathcal{F}}(\sum_{0 \leq i \leq M} pg^{-1} _{i} e_{i}  g_{i} \omega_{\underline{i}} ).
\end{align}

{\bf Notation.} In the following,  we let $n_{0}=0$ and  $d_{i}:=n_{i}-n_{i-1},$  for $i \geq 1.$  Suppose that $\underline{s}=(s_{1}, \cdots, s_{k}),$ where $s_{j}$ are positive integers, $\underline{i}:=(i_{1}, \cdots, i_{k}),$ where $1 \leq i_{j}\leq M,$ and  $\underline{\alpha}:=\{\alpha_{1},\cdots, \alpha_{r} \}  \subseteq \{n_{i} | 1 \leq i \leq k\} \cup \{d_{i}| 1 \leq i \leq k \}.$ Then  we let 
$$
S(\underline{s} ; \underline{i};\underline{\alpha})(z):=p^{\sum _{i}s_{i}}\sum \frac{z^{n_{k}}}{n_{1} ^{s_{1}} \cdots n_{k} ^{s_{k}}\zeta^{\underline{i_{1}} d_{1} +  \cdots + \underline{i_{k}}d_{k}}   },
$$
where the sum is taken over all $0<n_{1}<\cdots <n_{k},$ which satisfy both of the following two properties: 

(i) $p\not  |n_{1}$ and 

(ii) $p| \alpha_{i},$ for all $1 \leq i \leq r. $ 
 
If we take the sum over  all $0<n_{1}<\cdots <n_{k}$ which satisfy (ii) then we denote the resulting series by $T(\underline{s},\underline{i},\underline{\alpha})(z).$  We use $\underline{S}(\cdot)$  and $\underline{T}(\cdot)$ to denote $S(\cdot)$ and $T(\cdot)$ without the $p^{\sum _{i} s_{i}}$ factors.

%We have then, for example, $S(s;i)(z)=\sum _{ 0<n_{1} \atop { p \not \, | n_{1}}} \frac{z^{n_{1}}}{n_{1} ^{s} \zeta ^{\underline{i} n_{1}}},$ and  $$ S(s,t,r;i,j,k;d_{2},n_{3})(z)=\sum _{0<n_{1} <n_{2} <n_{3} \atop{p \not \, | n_{1} \atop { p | (n_{2}-n_{1}), \, p|n_{3} } }}\frac{z^{n_{3}}}{n_{1} ^{s} n_{2} ^{t} n_{3} ^{r} \zeta^{\underline{i} n_{1} +\underline{j} (n_{2}-n_{1})+\underline{k} (n_{3} -n_{2}) }}. $$

For any power series $f \in K[[ z]],$ we let $f[w]$ denote the coefficient of $z^{w}$ in $f.$

Let 
$$
F(\underline{s} ; \underline{i};\underline{\alpha})(n):=p^{\sum _{i}s_{i}}
\sum \frac{\zeta^{\underline{i_{1}}n_{1}+\cdots +\underline{i_{k}}n_{k}  }}{n_{1} ^{s_{1}} \cdots n_{k} ^{s_{k}} },
$$
where the sum is over all $0<n_{1}< \cdots <n_{k}<n$ that satisfy (i) and (ii) above. We denote the function obtained by taking the sum over $0<n_{1} < \cdots<n$ that satisfies (ii), by $G(\underline{s},\underline{i},\underline{\alpha})(n).$ Similarly, let $\underline{F}$ and $\underline{G}$ be the versions without the $p$-power factor. 
Clearly, 
$$
\underline{S}(\underline{s} ; \underline{i})[n]=\frac{\zeta^{-\underline{i_{k}} n} }{n^{s_{k}}}\underline{F}(\underline{s}',\underline{i}')(n),
$$
where $\underline{s}'=(s_{1},\cdots,s_{k-1})$ and $\underline{i}'=(i_{2}-i_{1},\cdots,i_{k}-i_{k-1}),$ and there are similar relations for any $\underline{\alpha}$ as above. 
 
Using the definition of $L^{(k)}$ for an $M$-power series function $L$ in Example \ref{exam main power} (iv), 
we define $F(s_{1},(s_{2});i,j)$ as follows. Noting that 
$$
F(s_{1},s_{2};i,j)(n)=p^{s_{2}}\sum_{0<k<n}\frac{F(s_{1};i)(k)}{k^{s_{2}}}\zeta^{\underline{j}k},
$$
we put 
$$
F(s_{1},(s_{2});i,j)(n)=p^{s_{2}}\sum_{0<k<n}F^{(s_{2})}(s_{1};i)(k)\zeta^{\underline{j}k}.
$$
We define $F(s_{1},(s_{2});i,j;\underline{\alpha})$ analogously, and we let $F^{(\cdot)} (\cdot):=F (\cdot) ^{(\cdot)}.$  

If we put $\underline{i}=(i,j,k),$ then we define $S(a,(b),c;\underline{i})$ and  $S(a,(b),(c);\underline{i})$ as follows:
$$
S(a,(b),c;\underline{i})[n]=\frac{p^{c}\zeta^{-\underline{k}n }}{n^{c}}F(a,(b);\underline{i}')(n)
$$
and 
$$
S(a,(b),(c);\underline{i})[n]=p^{c}\zeta^{-\underline{k}n }F^{(c)}(a,(b);\underline{i}')(n).
$$
We define $S(a,b,(c);\underline{i};\underline{\alpha})$ and $S(a,(b),c;\underline{i};\underline{\alpha})$ and $S(a,(b),(c);\underline{i};\underline{\alpha})$ with analogous identities.

When the limit exists, we let 
$
\mathcal{X}^{(\cdot)}(\cdot ):= \lim_{N \to \infty} F^{(\cdot)} (\cdot)(q^N).$

\section{Rigid analytic functions on $\mathcal{U}_{M}$}

Let $D(a,r)$ and $D(a,r)^{\circ}$ denote the closed and open   disks of radius $r$ around $a.$ Then $\mathcal{U}_{M} =\mathbb{P}^{1} _{K} \setminus \cup _{1 \leq i \leq M } D(\zeta^{i},1)^{\circ} . $  We will need the following proposition, which is a generalization of (Prop. 2, \cite{un}). 

\begin{proposition}\label{main rigid analytic}
Let $f$  be a rigid analytic function on $\mathcal{U}_{M}$ with $f(0)=0$ and  a power series expansion 
$$
f(z)=\sum_{0 <n}a_{n} z^{n}
$$
around 0. Then the sequence of rational functions
$$
f_{N}(z):=\frac{1}{1-z^{Mp^{N}}} \sum_{0<n \leq Mp^N}a_{n}z^{n} 
$$
converge uniformly on $\mathcal{U}_{M}$ to $f.$ The value of $f$ at $\infty$ is given by 
$$
f(\infty)=-\lim_{N \to \infty}a_{Mp^N}.
$$

\end{proposition}

\begin{proof}
Since $f$ is rigid analytic on the affinoid $\mathcal{U}_{M},$  it is a uniform limit of rational functions with poles outside $\mathcal{U}_{M}$ (\textsection 2.2, \cite{rigan}). We may also assume, without loss of generality, that these rational functions are 0 at 0.

\begin{claim}\label{claimpartialfractions}
If $r(z)$ is a rational function with poles outside $\mathcal{U}_{M},$ then $r(z)$ is a linear combination of functions of the form 
$$
\frac{z^{i}}{(1-az^M)^k},
$$
 for some $0 \leq k, $ $0\leq  i <M,$ and $|1-a|<1.$  
 \end{claim}

{\it Proof of the claim.} By the method of partial fractions, $r(z)$ is a linear combination of rational functions of the form 
\begin{eqnarray}\label{sprationalfunction}
\frac{1}{(1-bz)^{t}},
\end{eqnarray}
with $|b-\zeta^{i}|<1,$ for some $0 \leq i <M,$ and $0 \leq t.$ Therefore, we need to prove the statement only for  rational functions as in (\ref{sprationalfunction}).  Note that 
$$
\frac{1}{(1-bz)^{t}}=\frac{p(z)}{(1-az^{M})^{t}}=\sum _{0 \leq i \leq M-1}z^{i} \frac{q_{i}(1-az^M)}{(1-az^M)^t},
$$ 
for some  polynomials  $p(z)$ and $q_{i}(z),$ $0 \leq i \leq M-1$  and $a=b^{M}.$  Since $|1-a|<1,$ and the left hand side does not have a pole at $\infty,$ the right hand side is exactly as in the form stated in the claim. This proves the claim. 
\hfill $\Box$

Using the claim above, we will prove the following estimate on the coefficients of the Taylor expansion of $f:$

\begin{claim}\label{claimestimate} For $n \in \mathbb{N}:=\{1,2,3, \cdots \},$ let $n|_{N}$ denote the unique integer such that $0 <n|_{N} \leq Mp^N,$ and $Mp^N$  divides $n-n|_{N}.$ 
If we let $c_{N}:= \sup _{n \in \mathbb{N}} |a_{n} -a_{n|_{N}}|, $ then 
$$
\lim_{N \to \infty}c_{N}=0.
$$
\end{claim}
 
 {\it Proof of the claim.} First we note that, for $1 \leq k,$  $0\leq i<M$ and $|1-a|<1,$
 $$
 \frac{z^i}{(1-az^M)^k}= \sum_{0 \leq n} {n+k-1 \choose k-1}a^n z^{i+Mn}=:\sum _{ 0 \leq n} a_{n} z^{n};
 $$
 satisfy the property in the claim. 
 If $ n \not \equiv i (mod\, M) $ then $a_{n}=0$ and hence   
\begin{eqnarray*}
c_{N}&=&\sup _{n \in \mathbb{N} \atop {n \equiv i (mod \, M) }} |a_{n} -a_{n|_{N}}| \leq \sup_{s,t \geq 0} | a_{i+M(t+sp^N)}- a_{i+Mt}| \\
&=& \sup_{t \geq 0 \atop {s \geq 1} } | q(t+sp^N)a^{sp^N} - q(t) |=:d_{N},
\end{eqnarray*}
 where $q(t):={t+k-1 \choose k-1}$ is a polynomial of degree $k-1$ in  $t.$ Let $\alpha$ denote the maximum of the absolute value of the coefficients of $q(t),$ and $\beta:=|a-1|<1.$ Since $|a^{p}-1|\leq \max (\beta/p,\beta^p),$ choosing $N_{0}$ sufficiently large $|a^{p^{N_{0}}}-1|\leq p^{-1},$ and hence for $N \geq N_{0},$ $|a^{p^N} -1|\leq p^{-(N-N_{0})}.$ Then for $N \geq N_{0},$ 
 $
 d_{N}\leq \alpha (  p^{-N}+p^{-(N-N_0)}),
 $
and hence $\lim_{N \to \infty} d_{N}=\lim_{N \to \infty} c_{N}=0.$

  Since any rational function  $r(z),$ whose poles are outside $\mathcal{U}_{M},$ is a linear combination of functions as above (Claim \ref{claimpartialfractions}), the statement is true for $r(z).$ Note that for any power series $g(z):=\sum_{0 \leq n} b_{n}z^{n},$ which is convergent on $D(0,1)^{\circ}:$
 \begin{eqnarray}\label{modulusandsupremum}
 \sup_{0 \leq n} |b_{n}| \leq \sup_{|z|<1}|g(z)|.
 \end{eqnarray}
 Let $(r_{m})$ be a sequence of rational functions which are 0 at 0, have poles outside $\mathcal{U}_{M},$ and which converge, uniformly on $\mathcal{U}_{M},$ to $f.$ Letting 
 $$
 r_{m}(z):=\sum_{0<n} a_{n} ^{(m)}z^n,
 $$
 and $c_{N} ^{(m)}:=\sup_{n \in \mathbb{N}}|a_{n} ^{(m)}- a^{(m)} _{n|N}|;$ we know that $\lim_{N \to \infty} c_{N} ^{(m)}=0,$ for all $m.$ By uniform convergence and (\ref{modulusandsupremum}), $\lim_{m \to \infty} \sup_{N \in \mathbb{N}}|c_{N} ^{(m)}-c_{N}|=0.$  This implies the claim. 
 \hfill $\Box$
 
 Now, note that 
 $$
 f_{N+1}(z) -f_{N}(z)=\frac{1}{1-z^{Mp^{N+1}}}\sum_{0 <n \leq Mp^{N+1}}(a_{n}-a_{n|N})z^n.
 $$
Note that  $z \in \mathcal{U}_{M}$ if and only if $1 \leq |1-z^{M}|.$  Letting 
$0 < n \leq Mp^{N+1},$  
 $$|\frac{z^{n}}{1-z^{Mp^{N+1}}} |=|z^n|<1,$$ if $|z| <1;$ and  
 $$
 |\frac{z^{n}}{1-z^{Mp^{N+1}}} |\leq \frac{1}{|(1/z^{M})^{p^{N+1}}-1|} \leq 1,
 $$ 
  if $1\leq |z|$ and $z \in \mathcal{U}_{M}.$ Therefore, 
  $$
  \sup _{z \in \mathcal{U}_{M}} |f_{N+1}(z) -f_{N}(z)| \leq c_{N},
  $$
  and we conclude, by Claim \ref{claimestimate}, that $(f_{N})$ converges uniformly to a rigid analytic function on $\mathcal{U}_{M}.$ To see that this function, indeed,  is $f,$ we note that for $|z|<1,$ 
 $|f(z)-f_{N}(z) |\leq c_{N}.$ Then again Claim   \ref{claimestimate} implies the assertion. The last assertion follows from  $f_{N}(\infty)=-a_{Mp^N}.$ 
  \end{proof}
  
  \begin{corollary}\label{cor of main rigid}
  Let $f(z)=\sum _{0<n}a_{n}z^{n}$ be as in Proposition \ref{main rigid analytic}, and $0< l \leq pM$  then
   $$\lim_{N \to \infty}|a_{lq^{N+1}}-a_{lq^{N}}|=0.$$ If $\lim_{N \to \infty}lq^{N}a_{lq^{N}} $ exists then it is equal to 0. 
  \end{corollary}
  
  \begin{proof}
  Since $M|(q-1),$ $Mp^N | (lq^{N+1}-lq^N)$ and hence $|a_{lq^{(N+1)}}-a_{lq^{N}}|\leq 2c_{N}.$ Since $\lim_{N \to \infty}c_{N}=0,$ the first statement follows. Assume  that $\lim_{N \to \infty}lq^{N}a_{lq^{N}}=\alpha.$ Then 
  $$
  q\alpha=\lim _{N \to \infty} (lq^{N+1}a_{lq^{N}}+ lq^{N+1}(a_{lq^{N+1}} - a_{lq^{N}}))=\lim_{N \to \infty} lq^{N+1}a_{lq^{N+1}}=\alpha.
  $$
  Hence $\alpha=0.$
  \end{proof}

\section{Computation of the polylogarithmic part }

In this section we  determine the frobenius action on  the polylogarithmic quotient of the fundamental group of $X_{M} '.$

\subsection{Computation of $g_{j}[e_{0} ^{s}e_{i}]$}\label{Comp of 1}

Let $e_{\infty}\in {\rm Lie} \, \pi_{1,dR}(X_{M} ')$ denote the element which is obtained by ${\rm res}_{\infty},$ the residue at $\infty,$ as in \textsection\ref{connection on the fundamental torsor}.   

Applying $F_{*}$ to the identity
$$
 \sum_{0 \leq i \leq M}e_{i} +e_{\infty}=0,
$$
we get
$$
\sum_{0 \leq i \leq M} g_{i} ^{-1} e_{i} g_{i}=g_{\mathcal{F}}(\infty) ^{-1} (\sum_{0 \leq i \leq M}e_{i} )g_{\mathcal{F}}(\infty).
$$

Rewriting this, we obtain the fundamental identity 
\begin{eqnarray}\label{mainequality}
g_{\mathcal{F}}(\infty) (\sum_{0 \leq i \leq M} g_{i} ^{-1} e_{i} g_{i})=(\sum_{0 \leq i \leq M}e_{i} )g_{\mathcal{F}}(\infty).
\end{eqnarray}

From the equation (\ref{rewritingtheequation}), we obtain 
$$
dg_{\mathcal{F}}[e_{0}]=\mathcal{F}^* \omega_{0}-p\omega_{0}=0,
$$ 
and hence that $g_{\mathcal{F}} [e_{0}]=0,$ since $g_{\mathcal{F}}(0)=1.$  

Similarly, for $1 \leq i\leq M,$
$$
dg_{\mathcal{F}}[e_{i}]=\mathcal{F}^{*}\omega_{i}-p \omega_{\underline{i}},
$$
which gives that 
$$
g_{\mathcal{F}}(z)[e_{i}]=p\sum _{1 \leq n \atop {p \not \, | n}}\frac{(\zeta^{-\underline{i}}z)^{n} }{n},
$$
for $z \in D(0,1)^{\circ}.$ Since $g_{\mathcal{F}}[e_{i}]$ is a rigid analytic function Proposition \ref{main rigid analytic} implies that 
$$
g_{\mathcal{F}}(z)[e_{i}]=\lim_{N \to \infty}\frac{p}{1-z^{Mp^N}}\sum_{0<n \leq Mp^N \atop {p \not \, | n}}\frac{(\zeta^{-\underline{i}}z)^{n} }{n},
$$  
for $z \in \mathcal{U}_{M},$ and 
\begin{align}\label{g(infty)ei}
g_{\mathcal{F}}(\infty)[e_{i}]=0.
\end{align}
Comparing the coefficients of $e_{i}e_{0}$ in both sides of (\ref{mainequality}) gives 
$$
g_{\mathcal{F}}(\infty)[e_{i}]+g_{i}[e_{0}]=g_{\mathcal{F}}(\infty)[e_{0}].
$$
Using (\ref{g(infty)ei}), this gives 
$
g_{i}[e_{0}]=0.
$ That $g_{i}$ is group-like implies that
\begin{align}\label{gien}
g_{i}[e_{0} ^{n}]=0,
\end{align} 
 for all $0<n,$ and $0 \leq i \leq M.$
 
 Using this and  the equation (\ref{rewritingtheequation}) we see, by induction, that 
 
 \begin{eqnarray}\label{formula for gf}
g_{\mathcal{F}}(z)[e_{0} ^{s-1}e_{i}]=S(s;i)(z),
\end{eqnarray}
for $z \in D(0,1)^{\circ}.$ 
Note that if $\alpha $ is a group-like element with $\alpha[e_{0}]=0$ then 
\begin{align}\label{relofalpha}
\alpha[e_{0} ^{a} e_{i} e_{0} ^b]=(-1) ^{b} {a+b \choose a} \alpha[e_{0} ^{a+b} e_{i}].
\end{align}
This gives 
$$
g_{\mathcal{F}}(z)[e_{0} ^{a}e_{i} e_{0} ^{b}  ]=(-1) ^{b}{a+b \choose a} S(a+b+1;i).
$$
Using Proposition \ref{main rigid analytic} as above, we get 
\begin{eqnarray}\label{As a}
g_{\mathcal{F}}(\infty)[e_{0} ^{a} e_{i} e_{0} ^{b}]=0.
\end{eqnarray}

Using (\ref{rewritingtheequation}) and (\ref{gien}) we obtain that $
dg_{\mathcal{F}}[e_{i} e_{0} ^{s-1}e_{j}]=
$
$$
 \mathcal{F}^{*}\omega_{i} g_{\mathcal{F}}[e_{0} ^{s-1}e_{j}]-p\big(  g_{i}[e_{0} ^{s-1}e_{j}] \omega_{\underline{i}} + g_{\mathcal{F}}[e_{i} e_{0} ^{s-1}] \omega_{\underline{j}}+ g_{j} ^{-1}[e_{i}e_{0} ^{s-1}] \omega_{\underline{j}}       \big).
$$
From (\ref{As a}) and the fact that the above differential is regular at $\infty,$ we get
$$
g_{j} ^{-1} [e_{i} e_{0} ^{s-1}]=-g_{i}[e_{0} ^{s-1}e_{j}].
$$

Using this and solving the differential equation we obtain that 
$$
g_{\mathcal{F}}(z)[e_{i}e_{0} ^{s-1}e_{j}]=
$$
\begin{eqnarray*}
-S(s,1;j,i;d_{2})+(-1)^{s-1}S(s,1;i,j)+g_{i}[e_{0} ^{s-1}e_{j}] (T(1;i) -T(1;j)). 
\end{eqnarray*}

This gives that 
$$
g_{\mathcal{F}}(\infty)[e_{i}e_{0} ^{s-1}e_{j}]=(-1)^{s}p^{s+1}\lim_{N \to \infty} \frac{1}{Mp^N} \sum_{0<n< Mp^N \atop {p \not \, | n}}\frac{\zeta^{(\underline{j}-\underline{i})n}}{n^s}.
$$

Using this we find a formula for $g_{j}[e_{0} ^{s}e_{i}],$ with $s \geq 1,$ as follows. First upon comparing the coefficients of $e_{0}e_{i}e_{0} ^{s-1}e_{j}$ in (\ref{mainequality}) and using (\ref{gien}) and   (\ref{As a})    we find that 
$$
g_{\mathcal{F}}(\infty)[e_{i}e_{0} ^{s-1}e_{j}]=g_{j} ^{-1}[e_{0}e_{i}e_{0} ^{s-1}].
$$
Again by (\ref{gien}), $g_{j} ^{-1} [e_{0}e_{i}e_{0}^{s-1}]=-g_{j} [e_{0}e_{i}e_{0}^{s-1}]$ and by  (\ref{relofalpha}),  $g_{j} [e_{0}e_{i}e_{0}^{s-1}]=(-1)^{s-1}s g_{j}[e_{0} ^{s}e_{i}].$ 
Combining these we get  the following expression. 
\begin{proposition}\label{usual expression prop} 
For $s \geq 1,$ $$g_{j}[e_{0}^{s}e_{i}]= \frac{p^{s+1}}{s}\lim_{N \to \infty} \frac{1}{Mp^N} \sum_{0<n< Mp^N \atop {p \not \, | n}}\frac{\zeta^{(\underline{j}-\underline{i})n}}{n^s}
.$$ 
\end{proposition}

\subsection{An alternative expression for $g_{j}[e_{0} ^{s-1}e_{i}]$ when $i \neq j$}First note that by the expression for $g_{\mathcal{F}}[e_{i}e_{0}^{s-1}e_{j}]$ in  \textsection \ref{Comp of 1}
$$
q^{N}g_{\mathcal{F}}[e_{i}e_{0}^{s-1}e_{j}][q^{N}]=p((-1)^{s-1}\zeta^{-\underline{j}}F(s;j-i)(q^N)+g_{i}[e_{0}^{s-1}e_{j}](\zeta^{-\underline{i}}-\zeta^{-\underline{j}}))
$$

\begin{claim}
$\lim_{N \to \infty} q^{N}g_{\mathcal{F}}[e_{i}e_{0}^{s-1}e_{j}][q^{N}]=0$
\end{claim}

\begin{proof}
By Corollary \ref{cor of main rigid} it is enough to show that the above limit exists. This follows from the observation that 
$$
\sum_{0 < n < q^{N+1} \atop { p \not \,  | n }} \frac{\zeta ^{n}}{n^{s}} - \sum_{0 < n < q^{N} \atop { p \not \,  | n }} \frac{\zeta ^{n}}{n^{s}}=\sum _{1 \leq t \leq q-1} \sum_{0 < n < q^{N} \atop { p \not \,  | n }} \frac{\zeta ^{n+tq^N}}{(n+tq^N)^{s}}
$$
is congruent modulo $q^N$ to 
$$
\sum _{1 \leq t \leq q-1}\zeta^t \cdot \sum_{0 < n < q^{N} \atop { p \not \,  | n }}  \frac{\zeta ^{n}}{n^{s}}=0.
$$
\end{proof}
Then the above claim gives the following:
\begin{proposition}\label{alternative prop}
For  $i \neq j,$ we have 
$$
g_{i}[e_{0}^{s-1}e_{j}]=\frac{(-1)^{s-1}}{1-\zeta^{\underline{j}-\underline{i}}}\mathcal{X}(s;j-i).
$$
\end{proposition}

\subsection{Computation of $g_{j}[e_{i}]$}\label{Computation 5} By (\ref{alpha eq}), $g_{j}[e_{i}]=(\alpha_{j* }(g))[e_{i}]=g[e_{i-j}].$  Let  $\iota : X_{M} ' \to X_{1} '$ denote the inclusion. If $i=j$ then using the functoriality of frobenius wiht respect to $\iota$ we see that $g[e_{i-j}]$ computed on $X_{M} '$ is equal to $g[e_{1}]$ computed on $X_{1} '.$ But this last expression is 0 by  \cite[\textsection 5.6]{un}.  Suppose now that $i \neq j.$ Then $g_{j}[e_{i}]=\alpha_{i*}(g_{j-i})[e_{i}]=g_{j-i}[e_{M}].$  Then as above,  by the functoriality of frobenius for $\iota,$ $g_{j-i}[e_{M}],$ which is  computed on $X_{M} ',$ is equal to 
\begin{align}\label{expr}
(\, _{t_{0}}\gamma_{z'} \cdot F_{*} (\, _{z}\gamma_{t_{0}})) [e_{1}],
\end{align}
which is computed  on $X_{1} '.$ Here $z=\zeta^{\underline{j}-\underline{i}}$ and $z'=\zeta^{j-i}.$ Note that $\mathcal{F}$ is good lifting of frobenius on $\mathcal{U}_{1} \subseteq X_{1} '.$ Since $i \neq j,$ $z \in \mathcal{U}_{1},$ and since $\mathcal{F}(z)=z^{p}=z',$ we see that (\ref{expr}) is equal to $g_{\mathcal{F}}[\zeta^{\underline{j}-\underline{i}}][e_{1}].$ The last expression is computed by Proposition \ref{main rigid analytic} to be 
$$
\lim_{N \to \infty} \frac{p}{1-\zeta^{(\underline{j}-\underline{i})p^N }}\sum_{0< n< p^N \atop p \not \,  | n}\frac{\zeta^{(\underline{j}-\underline{i})n}}{n}.
$$
Therefore we have the following expression for $g_{j}[e_{i}].$
\begin{proposition}\label{prop 2}
If $i=j$ then $g_{j}[e_{i}]=0.$ Otherwise 
$$
g_{j}[e_{i}]=\log \frac{1-\zeta^{j-i}}{(1-\zeta^{\underline{j}-\underline{i}})^p}.
$$
\end{proposition}

\section{ $M$-power series functions} 

In order to compute the higher depth part of the frobenius action we first study the type of functions that appear in these computations. 

\begin{definition}
Let $n \in \mathbb{N}$ and  let $f: \mathbb{N}_{\geq n} \to \mathbb{Q}_{p}[\zeta]$ be any function. We say that $f$ is an {\it $M$-power series function}, if there exist power series $p_{i} (x) \in \mathbb{Q}_{p}[\zeta][[x]],$ which converge on $D(0,r_{i})$ for some $r_{i} > |p|$, for $0 < i  \leq pM,$ such that 
$
f(a)=p_{i}(a-i),
$ 
for all $a \geq n $ and $pM|(a-i).$ We let  the {\it absolute value of $f$} to be the maximum of the absolute values of the $p_{i}.$
 \end{definition}
 
\begin{remark}\label{rem main power}
(i) By the Weierstrass preparation theorem, the power series $p_{i}$ in the above definition are unique. 

(ii) Fix $0< l \leq pM,$ and let  $f$ be as above. Then there is a power series $p(x) \in \mathbb{Q}_{p}[\zeta][[x]] $ which converges on some $D(0,r)$ with $r>|p|$ and 

%Since  $pM | (q^{N}-q),$ there is  $1 \leq i \leq pM$ such that $pM |(q^{N} -i)$ for all $N.$ Therefore if $f$ is as above, then there is a power series $p(x) \in \mathbb{Q}_{p}[\zeta][[x]] $ which converges on some $D(0,r)$ with $r>|p|$ and 
$$
f(lq^N)=p(lq^N),
$$  
for $N$ sufficiently large. 

\end{remark}

\begin{example}\label{exam main power}{\rm

 (i) Let $s \in \mathbb{Z}$ and $f(k):=\zeta^{ik} k^{s},$ for $p \not  | k$ and $f(k)=0$ for $p|k. $      Then $f$ is an  $M$-power series function.

(ii) Clearly the sums and products of  $M$-power series functions are  $M$-power series functions. 

(iii) Let $f$ be an $M$-power series function. For any $0 <l \leq pM,$ with $p|l$ let 
$$
f_{l}:= \lim _{n \to 0 \atop {pM | (n-l)}}f(n).
$$
Let $f^{[1]}$ be defined by 
$$
f^{[1]}(k)=\frac{f(k)-f_{l}}{k},
$$
if $p|k$ and  $pM|(k-l);$ and $f^{[1]}(k)=0,$ if $p \not | k.$ 
We then see that  $f^{[1]}$ is an $M$-power series function. In fact, if $p|l,$ and $p$ is a power series around 0 such that $f(n)=p(n)$ for all $pM | (n-l)$ then $f^{[1]}(n)=q(n),$ for all $pM | (n-l),$ where   
$$
q(x)= \frac{p(x)-p(0)}{x}.
$$ 
Inductively, we let $f^{[k+1]}:=(f^{[k]})^{[1]}.$

(iv) Using the notation as above, let $f^{1}$ be defined by  $f^{(1)}(k):=f^{[1]}(k),$ if $p|k;$ and $f^{ (1)}(k)=\frac{f(k)}{k},$ if $p \not  |k.$ Then $f^{ (1)} $ is also an $M$-power series function. 

}
\end{example}

\begin{proposition}\label{prop main power}
Let $f: \mathbb{N}_{\geq n_{0}} \to \mathbb{Q}_{p}[\zeta]$ be an $M$-power series function. If we define $F:  \mathbb{N}_{\geq n_{0}} \to \mathbb{Q}_{p}[\zeta]$ by 
$$
F(n):= \sum _{n_{0}\leq  k\leq n}f(k)
$$
then $F$ is also an $M$-power series function. 
 \end{proposition}

\begin{proof}
Note that $f$ is uniquely extended to an $M$-power series function $\tilde{f}$ which is defined on all $\mathbb{N}.$  Then since $\tilde{F}(n):=\sum _{1 \leq k \leq n}\tilde{f}(k)=\tilde{F}(n_{0}-1)+F(n)$ for all $n \geq n_{0},$  that $\tilde{F}$ is an $M$-power series function implies the same for $F.$ Therefore, without loss of generality, we will assume that $n_{0} =1.$ 

For $1 \leq t \leq pM,$ let 
$$
F_{t}(n):=\sum_{1 \leq k \leq n \atop { pM | (k-t)}}f(k)=\sum_{0 \leq \alpha \leq n-t \atop {pM | \alpha}} p_{t} (\alpha),
$$
for $t \leq n$ and $F_{t}(n)=0$ otherwise. 

Since $F(n)=\sum _{1 \leq t \leq pM}F_{t}(n),$ it suffices to prove that each $F_{t}$ is an $M$-power series function. Fix $1 \leq i, t \leq pM$ and  suppose first that $t \leq i.$ Let $p_{t}(x)=\sum_{0 \leq j} a_{j} x^{j}.$ By assumption there is an $\varepsilon >0$ such that $\lim _{n \to \infty} a_{j} p^{j(1-\varepsilon)}=0.$

Recall the formula for the sum of the $j$-th powers:
\begin{eqnarray*}\label{sum of powers}
\sum_{1 \leq m \leq n} m^{j}=\frac{1}{j+1}\sum_{0 \leq k \leq j} {j+1 \choose k} (-1)^kB_{k} n^{j+1-k}
\end{eqnarray*}
where $B_{k}$ are the Bernoulli numbers defined by 
$$
\frac{x} {e^x-1}=\sum _{0 \leq k} \frac{B_{k}x^{k}}{k!}.
$$
The Von Staudt-Clausen theorem gives the bound $|B_{k}|\leq p.$

Then for $n\geq 0$
\begin{eqnarray*}
F_{t}(i+npM)&=&\sum_{0 \leq k \leq  j }\frac{a_{j}(pM)^{j}}{j+1}
 {j+1 \choose k}(-1)^{k}B_{k} n^{j+1-k}\\ 
 &=& \sum_{1 \leq l \atop {0 \leq k}}  \frac{  a_{l+k-1} (pM)^{k-1} }{l+k} {l+k \choose k}(-1)^k B_{k}(npM)^{l}.
\end{eqnarray*}
Therefore letting $q_{t}(x)= \sum _{1 \leq l} b_{l}x^{l},$ with 
$$
b_{l}=\sum _{0 \leq k} \frac{a_{l+k-1}(pM)^{k-1} }{l+k} {l+k \choose k} (-1)^{k}B_{k}, 
$$
we have $F_{t}(i+npM)=q_{t}(npM).$  

Note that  
$$
|b_{l} p^{l(1- \varepsilon /2)}| \leq p^{1+\varepsilon}{\rm max} _{k} |a_{l+k-1}p^{(l+k-1)(1-\varepsilon)} \frac{p^{(l/2 +k)\varepsilon} }{l+k}|.
$$ 
Since $\lim_{l \to \infty} a_{l+k-1}p^{(l+k-1)(1-\varepsilon)}=0$ and 
$$\lim_{l \to \infty}| \frac{p^{(l/2 +k)\varepsilon} }{l+k}|\leq \lim_{l \to \infty} | \frac{p^{(l +k)\frac{\varepsilon}{2}} }{l+k}| =0,$$
we see that $\lim_{l \to \infty}b_{l} p^{l(1- \varepsilon /2)}=0. $

On the other hand if $i < t,$ then  we have
$$
F_{t}(i+npM)=F_{t}(t+npM)-f(t+npM)=q_{t}(npM)-p_{t}(npM).
$$
This proves that $F_{t}$ is an $M$-power series function as desired. 
\end{proof}

\begin{corollary}\label{cor on M power}
If $f_{1}, f_{2} \cdots , f_{k}$ are $M$-power series functions, then the function $G$ defined by 
$$
G(n_{k}):= \sum _{0 <n_{1} <n_{2}< \cdots <n_{k}}f_{1}(n_{1} ) f_{2}(n_{2})\cdots f_{k}(n_{k})
$$
is an $M$-power series function. 
\end{corollary}

\section{Computation of the higher depth  part}

\subsection{Computation of $g_{i}[e_{j}e_{k}]$}\label{Comp 1} The regularity of $g_{\mathcal{F}}$  at $\infty$ implies  the regularity of $dg_{\mathcal{F}}[e_{i}e_{j}e_{k}]$ at $\infty.$ By (\ref{rewritingtheequation}), this gives  
\begin{align}\label{equality for 1}
-g_{\mathcal{F}}(\infty)[e_{j}e_{k}]+g_{\mathcal{F}}(\infty)[e_{i}e_{j}]+g_{k} ^{-1} [e_{i}e_{j}]+g_{j} ^{-1}[e_{i}]g_{j}[e_{k}]+g_{i}[e_{j}e_{k}]=0.
\end{align}

Since by \textsection \ref{Comp of 1},   
$$
g_{\mathcal{F}}(\infty)[e_{i}e_{j}]=-g_{j}[e_{0}e_{i}]=\frac{\mathcal{X}(2;i-j)}{1-\zeta^{\underline{i} -\underline{j}}}
$$
for $i\neq j,$ we have, for $i,j,$ $k$ all distinct, $g_{k} ^{-1} [e_{i}e_{j}]=$
\begin{align}\label{relation residue 2}
\frac{\mathcal{X}(2;j-k)}{1-\zeta^{\underline{j} -\underline{k}}}- \frac{\mathcal{X}(2;i-j)}{1-\zeta^{\underline{i} -\underline{j}}}+g_{j} [e_{i}]g_{j}[e_{k}]-g_{i}[e_{j}e_{k}].
\end{align}

Recall that  $$g_{\mathcal{F}}[e_{l}e_{m}]=
-S(1,1;m,l ;d_{2})+S(1,1; l,m)+g_{l}[e_{m}](T(1;l)-T(1;m) ).
$$

\subsubsection{Computation of $g_{i}[e_{j}e_{i}]$} Using (\ref{equality for 1}) with $k=i,$  the fact that $g_{\mathcal{F}}$ and $g_{l}$ are group-like  we obtain: 
$$
g_{i}[e_{j}e_{i}]=\frac{1}{2}g_{j}[e_{i}] ^{2}-g_{\mathcal{F}}(\infty)[e_{i}e_{j}].
$$

Then Proposition \ref{prop 2} gives the following expression. 
\begin{proposition}\label{prop 3}
Assuming that $i \neq j,$ we have  
\begin{align*}
g_{i}[e_{j}e_{i}]=\frac{1}{2}\big(\frac{\mathcal{X}(1;i-j) }{1-\zeta^{\underline{i} - \underline{j}}} \big) ^{2}-\frac{\mathcal{X}(2;j-i)}{1-\zeta^{\underline{j}-\underline{i}}}.
\end{align*}
\end{proposition}

\subsubsection{Computation of $g_{i}[e_{j}e_{k}]$} 

Using the differential equation above gives that $g_{\mathcal{F}}[e_{i}e_{j}e_{k}]=S(1,1,1;k, j,i;d_{2},d_{3})-S(1,1,1;j,k,i;d_{3})-S(1,1,1; j,i,k; d_{2})$
\begin{eqnarray*}
& &+S(1,1,1; i,j,k)-g_{j}[e_{k}](T(1,1;j,i;d_{2})-T(1,1;k,i; d_{2}))\\
& &+g_{i}[e_{j}](T(1,1;i,k)-T(1,1;j,k))-g_{k}[e_{j}]S(1,1;i,k)+g_{k} ^{-1}[e_{i}e_{j}]T(1,k)\\
& &+g_{j}[e_{k}]S(1,1;i,j)-g_{j}[e_{i}]g_{j}[e_{k}]T(1;j)+g_{i}[e_{j}e_{k}]T(1;i). 
\end{eqnarray*}
Let us group the terms as follows. 

(i) $q^N S(\underline{s};\underline{i};d_{2},d_{3})[q^N]=0,$ for any $\underline{s}:=(s_{1},s_{2},s_{3})$ and $\underline{i}:=(i_{1},i_{2},i_{3}).$

(ii) $q^{N} \underline{T}(1;i)[q^N]=\zeta^{-\underline{i}}$

(iii) $q^{N} \underline{S}(s,1;i,j)[q^N]=\zeta^{-\underline{j}} \underline{F}(s;j-i)(q^N).$ Hence the limit is 
$$
\lim_{N \to \infty} q^{N} \underline{S}(s,1;i,j)[q^N]=\zeta^{-\underline{j}}\underline{\mathcal{X}}(s;j-i).
$$

 (iv) $q^{N}\underline{S}(s_{1},s_{2},1;j,i,k;d_{2})[q^{N}]=\zeta^{-\underline{k}}\underline{F}(s_{1},s_{2};i-j,k-i;d_{2} )(q^N)$
 
 Hence 
 $
 \lim_{N \to \infty}q^{N}\underline{S}(s_{1},s_{2},1;j,i,k;d_{2})[q^{N}]=\zeta^{-\underline{k}}\underline{\mathcal{X}}(s_{1},s_{2};i-j,k-i;d_{2}).
 $

(v) $q^{N} (\underline{S}(1,1,1;j,k,i;d_{3})+\underline{g}_{j}[e_{k}](\underline{T}(1,1;j,i;d_{2})-\underline{T}(1,1;k,i; d_{2})  )  )[q^N]$

is equal to 
$$
\zeta^{-\underline{i}}(\underline{F}(1,1;k-j,i-k;n_{2})(q^N) +\underline{g}_{j}[e_{k}]((\underline{G}(1;i-j;n_{1})-\underline{G}(1;i-k;n_{1}))(q^N) )  )
$$
Let us rewrite the same expression as
$$
\zeta^{-\underline{i}}\sum_{0<n_{2} <q^N \atop {p| n_{2} } } \frac{\zeta^{n_{2}(\underline{i}-\underline{k})}}{n_{2}} \Big( \underline{g}_{j}[e_{k}](\zeta^{(\underline{k}-\underline{j})n_{2}}-1)+\underline{F}(1;k-j)(n_{2})  \Big).
$$

\begin{lemma}\label{lemma on F}
Let $0 < l \leq pM$ such that $p| l,$ and $1 \leq t$ then 
$$
\lim_{n \to 0 \atop {pM | (n-l)}} F(t;i )(n)= (-1)^{t-1} g[e_{0}^{t-1}e_{i}](1-\zeta^{\underline{i}l}).
$$
\end{lemma}

\begin{proof} 
By Corollary \ref{cor on M power}, $F(t;i )$ is an $M$-power series function. Therefore the limit exists and is equal to 
\begin{eqnarray*}
\lim _{N\to \infty} F(t;i)(lq^N)&=&p^t\lim_{N \to \infty} \sum_{0 < n_{1} < l q^N \atop {p \not \, | n_{1}} } \frac{\zeta^{\underline{i} n_{1}}}  {n_{1} ^t}=
p^t \sum _{0 \leq a  <l}  \zeta ^{\underline{i}a} \lim_{N \to \infty}  \sum_{0<n<q^N \atop { p \not \, | n}} \frac{\zeta^{\underline{i}n}}{n^t},
\end{eqnarray*} 
which is equal to $ g[e_{0} ^{t-1}e_{i}](1-\zeta^{\underline{i}l}),$ if $0<i<M,$ by Proposition \ref{alternative prop}. On the other hand, if $i=M,$ then the equality follows since both sides of the expressions are 0 \cite[\textsection 5.11]{un}. \end{proof}

The main expression is then equal to 
\begin{eqnarray*}
& &\zeta^{- \underline{i}}\sum_{0<n_{2} < q^N \atop {p| n_{2}} } \zeta ^{n_{2} (\underline{i} - \underline{k})}    F^{(1)}(1; k-j)(n_{2})=:\zeta^{-\underline{i}} F(1,(1);k-j,i-k; n_{2})(q^N). 
\end{eqnarray*}
 By Example \ref{exam main power} (iii),  $F^{(1)}(1; k-j ; n_{1})$ is an $M$-power series function. Then by Proposition \ref{prop main power}, and 
 Remark \ref{rem main power} (ii), the limit exists as $N \to \infty$ and is equal to 
 $
 \zeta^{-\underline{i}}\underline{\mathcal{X}}(1,(1);k-j,i-k;n_{2}).
 $

(vi) $q^{N} (\underline{S}(1,1,1;i,j,k)+\underline{g}_{i}(e_{j})(\underline{T}(1,1;i,k)-\underline{T}(1,1;j,k)))[q^{N}]$

\begin{eqnarray*}
&=&\zeta^{-\underline{k}} \sum _{0<n_{2}<q^N} \frac{ \zeta ^{(\underline{k}-\underline{j})n_{2}}}{n_{2}}(\underline{g}_{i}[e_j] (\zeta^{(\underline{j}-\underline{i})n_{2}} -1)+  \underline{F}(1;j-i))\\
&=& \zeta^{-\underline{k}}\sum _{ 0< n <q^{N}} \zeta^{(\underline{k} -\underline{j})n}\underline{F}^{(1)}(1;j-i)(n)\\
&=&\zeta^{-\underline{k}}\underline{F}(1,(1);j-i,k-j)(q^N)
\end{eqnarray*}
As above the limit of the main expression as $N\to \infty$  exists and is equal to $\zeta^{-\underline{k}} \underline{\mathcal{X}}(1,(1);j-i,k-j).$

Therefore we have, $g_{i}[e_{j}e_{k}]=-\zeta^{\underline{i}-\underline{k}}g_{k} ^{-1}[e_{i}e_{j}]+\zeta^{\underline{i}-\underline{j}} g_{j}[e_{i}]g_{j}[e_{k}] $
\begin{align*}
& &-\zeta^{\underline{i}-\underline{j}} \mathcal{X}(1;j-i)g_{j}[e_{k}]
+\zeta^{\underline{i}-\underline{k}} \mathcal{X}(1;k-i)g_{k}[e_{j}]+\zeta^{\underline{i}-\underline{k}}\mathcal{X}(1,1;i-j, k-i; d_{2} )\\
& &+\mathcal{X}(1,(1);k-j,i-k;n_{2})-\zeta^{\underline{i}-\underline{k}}\mathcal{X}(1,(1);j-i,k-j).
\end{align*}

Using equation (\ref{relation residue 2}), we see that, for $i,j,$ $k$ distinct: 
$$(1-\zeta^{\underline{i}-\underline{k}})g_{i}[e_{j}e_{k}]=$$
\begin{eqnarray*}
& &\mathcal{X}(1,(1);k-j,i-k;n_{2})-\zeta^{\underline{i}-\underline{k}}\mathcal{X}(1,(1);j-i,k-j)\\
& &+\zeta^{\underline{i}-\underline{k}}\mathcal{X}(1,1;i-j, k-i; d_{2} )-\frac{\zeta^{\underline{i}-\underline{k}}\mathcal{X}(2;j-k)}{1-\zeta^{\underline{j} -\underline{k}}}+ 
\frac{\zeta^{\underline{i}-\underline{k}} \mathcal{X}(2;i-j)}{1-\zeta^{\underline{i} -\underline{j}}}
\\
& &-\frac{\zeta^{\underline{i}-\underline{j}}\mathcal{X}(1;j-i)\mathcal{X}(1;k-j)}{1-\zeta^{\underline{k} -\underline{j}}} +\frac{\zeta^{\underline{i}-\underline{k}}\mathcal{X}(1;j-k) \mathcal{X}(1;k-i)}{1-\zeta^{\underline{j}-\underline{k}}}\\
& & +\frac{(\zeta^{\underline{i}-\underline{j}}-\zeta^{\underline{i}-\underline{k}})\mathcal{X}(1;i-j)\mathcal{X}(1;k-j)}{(1- \zeta^{\underline{i}-\underline{j}})(1- \zeta^{\underline{k}-\underline{j}})}.\\
\end{eqnarray*}

\begin{corollary}\label{formula for gz 5} We have $g_{\mathcal{F}}[e_{i}e_{j}e_{k}]=$
\begin{align*}
& &S(1,1,1;k, j,i;d_{2},d_{3})-S(1,(1),(1);j,k,i;d_{3})-S(1,1,(1); j,i,k;d_{2})\\
& &+S(1,(1),(1); i,j,k)
-g_{k}[e_{j}]S (1,(1);i,k)+g_{j}[e_{k}]S(1,(1);i,j),
\end{align*}
where $g_{k}[e_{j}]$ is given by  Proposition \ref{alternative prop}. 
\end{corollary}

\subsection{An explicit  formula for  $F(t;i)$ } 

The following very simple proposition will be crucial throughout the paper. 
\begin{proposition}\label{fun prop}
Let $f$  be a  rigid analytic function on $\mathcal{U}_{M}$ such that
$$
df= dg+ h \omega_{0}+\sum _{1 \leq i \leq M}\alpha_{i} \omega_{i},
$$
with $g(z)=\sum _{0<n} a_{n}z^{n},$ $h(z)=\sum_{0<n}b_{n}z^n,$ $\alpha_{i} \in \mathbb{C}_{p}.$   Suppose that for $0<l \leq pM,$   $\lim_{N\to \infty} (lq^Na_{lq^N}+b_{lq^N}) $ exists.  Then 
$$
\lim_{N \to \infty}(lq^Na_{lq^N}+ b_{lq^{N}})=\sum _{1 \leq i \leq M}\zeta^{-il}\alpha_{i}.
$$
\end{proposition}

\begin{proof}
This is an immediate consequence of Corollary \ref{cor of main rigid}. 
\end{proof}

Suppose that $t\geq 1$ and  $s\geq 2,$ the equation (\ref{rewritingtheequation}) gives that 
$$
dg_{\mathcal{F}}[e_{j}e_{0} ^{t-1}e_{k}e_{0} ^{s-1}]=\mathcal{F}^*\omega_{j}g_{\mathcal{F}}[e_{0} ^{t-1}e_{k}e_{0} ^{s-1}]-pg_{\mathcal{F}}[e_{j}e_{0} ^{t-1}e_{k}e_{0} ^{s-2}]\omega_{0}-pg_{j}[e_{0} ^{t-1}e_{k}e_{0} ^{s-1}]\omega_{\underline{j}}.
$$
Since 
%$$
%\underline{g}_{\mathcal{F}}(z)[e_{k}e_{0} ^{s-1}]=(-1)^{s-1} \sum _{0 <n \atop {p \not \, | n}}\frac{(\zeta^{-\underline{k}} z)^{n}}{n^{s}},
%$$
$g_{\mathcal{F}}[e_{0} ^{t-1}e_{k}e_{0} ^{s-1}]=(-1)^{s-1}{s+t-2 \choose t-1}S(s+t-1;k) $
by (\ref{formula for gf}), and letting $(x)_{k}:=x(x+1)\cdots (x+(k-1)),$ we see that 
$$
\frac{(-1)^{s}}{(t-1)!}(s)_{t-1}dS(s+t-1,1;k,j;d_{2})=\mathcal{F}^{*}\omega_{j}g_{\mathcal{F}}[e_{0}^ {t-1}e_{k}e_{0}^{s-1}].
$$ 

This gives that $dg_{\mathcal{F}}[e_{j}e_{0} ^{t-1}e_{k}e_{0} ^{s-1}]=\frac{(-1)^{s}}{(t-1)!}(s)_{t-1}dS(s+t-1,1;k,j;d_{2})$
\begin{eqnarray*}
-pg_{\mathcal{F}}[e_{j}e_{0} ^{t-1}e_{k}e_{0} ^{s-2}]\omega_{0}-pg_{j}[e_{0} ^{t-1}e_{k}e_{0} ^{s-1}]\omega_{\underline{j}}.
\end{eqnarray*}

\begin{proposition}
If we let $g_{\mathcal{F}}(z)[e_{j}e_{0} ^{t-1}e_{k}e_{0} ^{s-1}]=\sum _{0 < n}c_{n}z^{n},$ then 
$$
\lim_{N\to \infty }c_{lq^N}=\zeta^{-\underline{jl}}g_{j}[e_{0} ^{t-1}e_{k}e_{0} ^{s}],
$$
for any $0 <l\leq pM.$
\end{proposition}

\begin{proof}
We will prove the proposition by induction on $s.$ Note that $g_{\mathcal{F}}[e_{j}e_{0} ^{t-1}e_{k}]=$
$$
-S(t,1;k,j;d_{2})+(-1)^{t-1}S(t,1;j,k)+g_{j}[e_{0} ^{t-1}e_{k}](T(1;j)-T(1;k) ), 
$$
by  \textsection \ref{Comp of 1}.  The coefficient of $z^{n}$ in $\underline{g}_{\mathcal{F}}[e_{j}e_{0} ^{t-1}e_{k}]$ is $\frac{K(n)}{n}$ where  $K(n):=$
$$
-\zeta^{-n \underline{j}}\underline{F}(t;j-k; d_{2})(n)+(-1)^{t-1}\zeta^{-\underline{k}n}\underline{F}(t;k-j)(n)+\underline{g}_{j}[e_{0} ^{t-1}e_{k}](\zeta^{-\underline{j}n}-\zeta^{-\underline{k}n}).
$$

By Lemma \ref{lemma on F} we see that 
$$
\lim_{N \to \infty}F(t;k-j)(lq^N)=(-1)^{t-1}g_{j}[e_{0} ^{t-1}e_{k}](1-\zeta^{(\underline{k} -\underline{j})l}).
$$
 This implies by Example \ref{exam main power} (iv) that $\frac{K(n)}{n}=$
$$
K^{ ( 1)}(n)=-\zeta^{-n\underline{j}}\underline{F}(t;j-k; d_{2})(n)+\zeta^{-\underline{k}n}\underline{F}^{(1)}(t;k-j)(n) ,$$ is an $M$-power series function and hence by Remark  \ref{rem main power} (ii), 
$$
\lim _{N \to \infty} K^{( 1) }(lq^{N})=\lim _{N \to \infty} \zeta^{-\underline{k}}\underline{F}^{(1)}(t;k-j)(lq^N)
$$ exists, for any $0<l\leq pM.$

%Now note that $S(1,1;k,j;n_{1}{\rm :} d_{2})[q^N]=0$ and that 
%\begin{eqnarray*}
%& &\lim_{N \to \infty} q^{N}(S(1,1;j,k;n_{1})+g_{j}[e_{k}](S(1;j)-S(1;k)))[q^N]\\
%& &= \zeta^{-\underline{k}} \lim_{N\to \infty} F(1;k-j;n_{1})(q^N)+  g_{j}[e_{k}](\zeta^{-\underline{j}} -\zeta^{-\underline{k}})=0
%\end{eqnarray*}
%by \textsection  \ref{Computation 5} and Lemma \ref{lemma on F}. Therefore $\lim_{N \to \infty} g_{\mathcal{F}}[e_{j}e_{k}][q^{N}]=0.$ 

Since $dg_{\mathcal{F}}[e_{j}e_{0} ^{t-1}e_{k}e_{0} ]=
tdS(t+1,1;k,j;d_{2})-pg_{\mathcal{F}}[e_{j}e_{0} ^{t-1}e_{k}]\omega_{0}-pg_{j}[e_{0}^{t-1}e_{k}e_{0} ]\omega_{\underline{j}},$ 
the existence of the above limit and Proposition \ref{fun prop} implies that 
$$
\lim_{N \to \infty} K^{ ( 1)}(lq^{N})=  \lim _{N \to \infty} \zeta^{-\underline{k}}\underline{F}^{( 1)}(t;k-j)(lq^N)=\zeta^{-\underline{j}l}\underline{g}_{j}[e_{0} ^{t-1}e_{k}e_{0}].
$$
This gives that $g_{\mathcal{F}}[e_{j}e_{0} ^{t-1}e_{k}e_{0}]=$
$$
\frac{1}{(t-1)!}\sum_{0 \leq r \leq 1}(r+1)_{t-1}S(t+r,2-r;k,j;d_{2}) +(-1)^t S(t,(2);j,k),
$$
and the coefficient of $z^{n}$ in this expression is
\begin{align*}
& &\frac{\zeta^{-n\underline{j}}}{(t-1)!}\sum _{0 \leq r \leq 1}   (r+1)_{t-1} \frac{F(t+r;j-k;d_{2})(n)}{n^{2-r}}+(-1)^t\zeta^{-n\underline{k}}F^{ (2)}(t;k-j;n_{1})(n).\\
\end{align*}
Since $\frac{F(t;i; d_{2})}{n^{k}}=F^{ (k)}(t;i; d_{2}), $
we inductively we arrive at the following expression for $g_{\mathcal{F}}[e_{j}e_{0} ^{t-1}e_{k}e_{0}^{s-1}]=$
\begin{eqnarray*}
\frac{(-1)^{s}}{(t-1)!}\sum_{0 \leq r \leq s-1}(r+1)_{t-1}S(t+r,s-r;k,j;d_{2})
+(-1)^{s+t}S(t,(s);j,k),
\end{eqnarray*}
and the coefficient of $z^{n}$ in this expression is 
$$
\frac{(-1)^{s}\zeta^{-n\underline{j}}}{(t-1)!} \sum _{0 \leq r \leq s-1}F^ { (s-r)}(t+r;j-k;d_{2})(n)+
(-1)^{s+t}\zeta^{-n\underline{k}}F^{ ( s)}(t;k-j)(n).
$$

 Now let $g_{\mathcal{F}}[e_{j}e_{0} ^{t-1}e_{k}e_{0} ^{s-1}]=\sum c_{n}z^{n}.$ Since $c_{n}$ is expressed in terms of the values of an $M$-power series function by the above expression, the limit $\lim_{N \to \infty} c_{lq^N}$ exists. In order to find this limit, we employ Proposition \ref{fun prop} in the differential equation for $dg_{\mathcal{F}}[e_{j} e_{0} ^{t-1}e_{k} e_{0} ^{s}]$ and find that 
$$
\lim_{N \to \infty} c_{lq^N}=\zeta ^{-\underline{j}l}g_{j}[e_{0} ^{t-1}e_{k}e_{0}^{s}].
 $$

Using the expression above this gives $\lim_{N \to \infty} F^{ (s)}(t;k-j)(lq^N)=$
$$
(-1)^{s+t}\zeta^{(\underline{k}-\underline{j} )l}g_{j}[e_{0} ^{t-1}e_{k}e_{0} ^{s}]=\frac{(-1)^{t}}{(t-1)!}\zeta^{(\underline{k}-\underline{j} )l}(s+1)_{t-1}g_{j}[e_{0} ^{s+t-1} e_{k}]. 
$$

These limits determine the $M$-power series function $F(t;i)$ completely as 
$$F(t;i)(n)=
(-1)^{t-1} (g[e_{0} ^{t-1}e_{i}](1-\zeta^{\underline{i}n} )-\frac{1}{(t-1)!} \sum _{1\leq r} \zeta^{\underline{i} n}(r+1)_{t-1}g[e_{0}^{r+t-1}e_{i}]n^r),
$$
for $p|n.$
\end{proof}
The proof of the above proposition has the following corollaries. 

\begin{corollary}
For $t\geq 1,$  $p|n$ and $M \not |i,$ we have $$p^t\sum _{0<k<n \atop {p \not \, | k}} \frac{\zeta^{\underline{i}k}}{k^t}=(-1)^{t-1} (g[e_{0} ^{t-1}e_{i}](1-\zeta^{\underline{i}n} )-\frac{1}{(t-1)!} \sum _{1\leq r} \zeta^{\underline{i} n}(r+1)_{t-1}g[e_{0}^{r+t-1}e_{i}]n^r).
$$
\end{corollary}

\begin{corollary}\label{formula for gz 6}
For, $t,s \geq 1$ and $j \neq k,$ we have $g_{\mathcal{F}}[e_{j}e_{0} ^{t-1}e_{k}e_{0}^{s-1}]=$
\begin{eqnarray*}
\frac{(-1)^{s}}{(t-1)!}\sum_{0 \leq r \leq s-1}(r+1)_{t-1}S(t+r,s-r;k,j;d_{2})
+(-1)^{s+t}S(t,(s);j,k).
\end{eqnarray*}
\end{corollary}

\subsection{Computation of $g_{i} [e_{j} e_{k}e_0 ^{s}] $ } We  already made this computation for $s=0$ in  \textsection \ref{Comp 1}. We will do induction on $s.$ So we assume that $s>0.$

The differential equation gives $dg_{\mathcal{F}}[e_{i}e_{j}e_{k}e_{0} ^s]=g_{\mathcal{F}}[e_{j}e_{k}e_{0} ^s]\mathcal{F}^*\omega_{i}$
\begin{align*}
-pg_{\mathcal{F}}[e_{i}e_{j}e_{k}e_{0} ^{s-1} ]\omega_{0} -pg_{i}[e_{j}e_{k}e_{0} ^s]\omega_{\underline{i}}
-pg_{j} ^{-1}[e_{i}]g_{j}[e_{k}e_{0} ^s]\omega_{\underline{j}}-pg_{\mathcal{F}}[e_{i}]g_{j}[e_{k}e_{0} ^s]\omega_{\underline{j}}.
 \end{align*}
 
 We will use Proposition \ref{fun prop} to make the computation.  Using Corollary \ref{formula for gz 6}, $g_{\mathcal{F}}[e_{j}e_{k}e_{0}]\mathcal{F}^*\omega_{i}=
 d(-\sum _{0 \leq r \leq 1} S(1+r,2-r,1;k,j,i;d_{2},d_{3})+S(1,(2),1;j,k,i;d_{3}) )$
 and $-g_{\mathcal{F}}[e_{i}]\omega_{\underline{j}}=
 dS(1,1;i,j).$

 Note that $lq^N S(a,b,1;k,j,i;d_{2},d_{3})[lq^N]=0,$ 
 $$
 lq^N S(a,(b),1;j,k,i;d_{3}) )[lq^N]=\zeta^{-\underline{i}l}F(a,(b); k-j, i-k;n_{2})(lq^N)
 $$
and 
 $$
 lq^{N} S(a,1;i,j)[lq^N]=\zeta^{-\underline{j}l}F(a;j-i)(lq^N).
 $$ 
  
On the other hand using  Corollary \ref{formula for gz 5} we obtain an expression for $g_{\mathcal{F}}[e_{i}e_{j} e_{k}],$ and noting that:

(i) $S(a,(b),(1);j,k,i;d_{3})[lq^n]=\zeta^{-\underline{i}l}F^{(1)}(a,(b);k-j,i-k;n_{2})(lq^N)$ 

(ii) $S(a,b,(1); j,i,k;d_{2})[lq^N]=\zeta^{-\underline{k} l} F^{(1)}(a,b;i-j,k-i;d_{2})(lq^N)$

(iii) $S(a,(b),(1); i,j,k)[lq^N]=\zeta^{-\underline{k}l}F^{(1)} (a,(b);j-i,k-j)(lq^N)$

(iv) $S(a,(1);i,k)[lq^N]=\zeta^{-\underline{k}l} F^{ (1)}(a;k-i)(lq^N).$  

Since the above limits exist as $N \to \infty$ we can use Proposition \ref{fun prop} and obtain that the limit of the following as $N \to \infty$ is equal to 
$
 -g_{i}[e_{j}e_{k}e_{0} ]\zeta^{-\underline{i}l}
-g_{j} ^{-1}[e_{i}]g_{j}[e_{k}e_{0} ]\zeta^{-\underline{j}l}:
$
\begin{align*}
& &\zeta^{-\underline{i}l}F(1,(2); k-j, i-k;n_{2})(lq^N)+g_{j}[e_{k}e_{0}]\zeta^{-\underline{j}l}F(1;j-i)(lq^N)\\
& &+\zeta^{-\underline{i}l}F^{(1)}(1,(1);k-j,i-k;n_{2})(lq^N)+\zeta^{-\underline{k} l} F^{ (1)}(1,1;i-j,k-i;d_{2})(lq^N)\\
&& -\zeta^{-\underline{k}l}F^{ (1)} (1,(1);j-i,k-j)(lq^N)+g_{k}[e_{j}]\zeta^{-\underline{k}l} F^{ (1)}(1;k-i)(lq^N)\\
& &-g_{j}[e_{k}]\zeta^{-\underline{j}l} F^{ ( 1)}(1;j-i)(lq^N).
\end{align*}
 
 This gives the following formula, with $i,j,$ and $k$ pairwise distinct, for $g_{i}[e_{j}e_{k}e_{0}]:$
 \begin{align*}
& &-\mathcal{X}(1,(2); k-j, i-k;n_{2})-\frac{\zeta^{\underline{i}-\underline{j}}}{1-\zeta^{\underline{k} -\underline{j}}}\mathcal{X}(2;k-j)\mathcal{X}(1;j-i)\\
& &-\mathcal{X}^{(1)}(1,(1);k-j,i-k;n_{2})-\zeta^{\underline{i}-\underline{k} } \mathcal{X}^{ (1)}(1,1;i-j,k-i;d_{2})\\
&& +\zeta^{\underline{i}-\underline{k}}\mathcal{X}^{ (1)} (1,(1);j-i,k-j)+\frac{\zeta^{\underline{i}-\underline{j}}\mathcal{X}(1;i-j)\mathcal{X}(2;k-j)}{(1-\zeta^{\underline{i} -\underline{j}})(1-\zeta^{\underline{k}-\underline{j}})}\\
& &-\frac{\zeta^{\underline{i}-\underline{k}}\mathcal{X}(1;j-k)\mathcal{X}^{ (1)}(1;k-i)}{1-\zeta^{\underline{j} -\underline{k}}}+\frac{\zeta^{\underline{i}-\underline{j}}\mathcal{X}(1;k-j)
\mathcal{X}^{ (1)}(1;j-i)}{1-\zeta^{\underline{k}-\underline{j}}}.
 \end{align*}
 
Now  Proposition \ref{fun prop} implies  that $g_{\mathcal{F}}[e_{i}e_{j}e_{k}e_{0}]=$ 
\begin{align*}
& &-\sum_{0 \leq p, q,r\leq 1 \atop {p+q+r=1}}S(1+p,1+q,1+r;k,j,i;d_{2},d_{3}) +S(1 , ( 2),(1) ;j,k,i;d_{3} )\\
& &+S(1, ( 1), (2);j,k,i; d_{3} )+ S(1,1,(2);j,i,k;d_{2}) +g_{j}[e_{k}e_{0}]S(1,(1);i,j)\\
&  &-S(1,(1),(2);i,j,k) +g_{k}[e_{j}]S(1,(2);i,k) -g_{j}[e_{k}]S(1,(2);i,j). 
\end{align*}

We use this information in the differential equation for $dg_{\mathcal{F}}[e_{i}e_{j}e_{k}e_{0}^{2}]$ above and this gives a formula for $g_{i}[e_{j} e_{k}e_{0} ^{2}].$ Inducting on $s,$ we find the following formulas for $g_{\mathcal{F}}[e_{i}e_{j}e_{k} e_{0} ^{s-1}]$ and $g_{i} [e_{j} e_{k} e_{0} ^{s}].$ Namely, $g_{\mathcal{F}}[e_{i}e_{j} e_{k}e_{0} ^{s-1}]=$
\begin{eqnarray*}
& &(-1)^{s-1}\sum_{0 \leq p,q,r \atop {p+q+r=s-1}}S(1+p,1+q,1+r;k,j,i; d_{2},d_{3})+(-1)^{s}g_{k}[e_{j}] S(1,(s);i,k) \\
& &+(-1)^{s}\sum_{0 \leq p,q \atop {p+q=s-1}}S(1,(1+p),(1+q); j,k,i;d_{3})+(-1)^{s}S(1,1,(s);j,i,k;d_{2} )\\
& &+\sum_{0 \leq r \leq s-1}(-1)^{r}g_{j}[e_{k}e_{0} ^{s-1-r}]S(1,(1+r);i,j) +(-1)^{s-1}S(1,(1), (s);i,j,k ).
\end{eqnarray*}

Then using the differential equation for $dg_{\mathcal{F}}[e_{i}e_{j}e_{k}e_{0}^{s}]$ and using Proposition \ref{fun prop} and noting that

(i) $S(a,(b);i,k) $ contributes  $\zeta^{- \underline{k}}\mathcal{X}^{( b)}(a;k-i)$

(ii) $S(a,(b),(c); j,k,i;d_{3})$ contributes $\zeta^{-\underline{i}}\mathcal{X}^{ (c)}(a,(b);k-j, i-k;n_{2})$

(iii) $S(a,b,(c);j,i,k;d_{2} )$ contributes $\zeta^{-\underline{k}}\mathcal{X} ^{ (c)}(a,b;i-j,k-j;d_{2})$

(iv) $S(a,(b);i,j)$ contributes $\zeta^{-\underline{j}}\mathcal{X} ^{(b)}(a;j-i)$

(v) $S(a,(b), (c);i,j,k )$
 contributes $\zeta^{-\underline{k}}\mathcal{X}^{(c)}(a,(b);j-i,k-j).$ 
 
 Therefore we obtain that $-\zeta^{-\underline{i}}g_{i}[e_{j}e_{k}e_{0} ^{s}]-\zeta^{-\underline{j}}g_{j} ^{-1}[e_{i}]g_{j}[e_{k}e_{0} ^{s}]=$ 
 \begin{eqnarray*}
& & (-1)^{s+1}\zeta^{-\underline{k}}g_{k}[e_{j}]\mathcal{X}^{(s)}(1;k-i)
+ (-1)^{s+1}\zeta^{-\underline{k}} \mathcal{X}^{(s)}(1,1;i-j,k-j{\rm :}d_{2})\\
& &+(-1)^{s+1}\zeta^{-\underline{i}} \sum _{0\leq r \leq s}\mathcal{X}^{(r)}(1, (s+1-r);k-j,i-k;n_{2}) \\
& &+\zeta^{-\underline{j}}\sum_{0 \leq r \leq s}(-1)^{r}g_{j}[e_{k}e_{0}^{s-r}] \mathcal{X}^{(r)}(1;j-i)
+(-1)^{s}\zeta^{-\underline{k}} \mathcal{X}^{(s)}(1,(1);j-i,k-j).
 \end{eqnarray*}

This proves the following proposition.

\begin{proposition}\label{comp of gijk0}
Assume that $i,j$ and $k$ are pairwise distinct and that $s>0.$ Then we have $g_{i}[e_{j}e_{k}e_{0} ^{s}]=$
\begin{eqnarray*}
& & (-1)^{s}\zeta^{\underline{i}-\underline{k}}\frac{\mathcal{X}(1;j-k)}{1-\zeta^{\underline{j}-\underline{k}}}\mathcal{X}^{(s)}(1;k-i)
+ (-1)^{s}\zeta^{\underline{i}-\underline{k}} \mathcal{X}^{(s)}(1,1;i-j,k-j;d_{2})\\
& &+(-1)^{s}\sum _{0\leq r \leq s}\mathcal{X}^{(r)}(1, (s+1-r);k-j,i-k;n_{2})\\
& &-\frac{\zeta^{\underline{i}-\underline{j}}}{1-\zeta^{\underline{k}-\underline{j}}}\sum_{0 \leq r \leq s}(-1)^{r}\mathcal{X}(s-r+1;k-j)\mathcal{X}^{(r)}(1;j-i)\\
& &+(-1)^{s+1}\zeta^{\underline{i}-\underline{k}} \mathcal{X}^{(s)}(1,(1);j-i,k-j)+\zeta^{\underline{i}-\underline{j}}\frac{\mathcal{X}(1,i-j)}{1-\zeta^{\underline{i}-\underline{j} }}\frac{\mathcal{X}(s+1;k-j)}{1-\zeta^{\underline{k}-\underline{j}}}.
 \end{eqnarray*}
\end{proposition}

 \subsection{Computation of $g_{i}[e_{j} e_{0} ^{s-1}e_{j}e_{0}^{t-1}]$}
 
We know the answer if $s=1$ from the previous section.  
 Let us assume that $s>1.$ Also first assume that $t=1.$

 \subsubsection{Computation of $g_{i}[e_{j} e_{0} ^{s-1}e_{j}]$} Assume that $s>1.$ The differential equation gives $ dg_{\mathcal{F}}[e_{i}e_{j}e_{0} ^{s-1}e_{k}]=$

 \begin{align*}
 & &g_{\mathcal{F}}[e_{j}e_{0} ^{s-1}e_{k}]\mathcal{F}^{*}\omega_{i}-pg_{i}[e_{j}e_{0} ^{s-1}e_{k}]\omega_{\underline{i}}-pg_{\mathcal{F}}[e_{i}]g_{j}[e_{0} ^{s-1}e_{k}]\omega_{\underline{j}}
 -pg_{j} ^{-1}[e_{i}]g_{j}[e_{0} ^{s-1}e_{k}]\omega_{\underline{j}}\\
 & & -pg_{\mathcal{F}}[e_{i}e_{j}e_{0} ^{s-1}]\omega_{\underline{k}}-pg_{\mathcal{F}}[e_{i}]g_{k} ^{-1}[e_{j}e_{0} ^{s-1}]\omega_{\underline{k}}-pg_{k} ^{-1}[e_{i}e_{j}e_{0} ^{s-1}] \omega_{\underline{k}}.
  \end{align*}
 
 Computing residues at $\infty$ in the above expression, we obtain that $g_{i}[e_{j}e_{0} ^{s-1}e_{k}]=$
\begin{align*} 
g_{\mathcal{F}}(\infty)[e_{j}e_{0} ^{s-1}e_{k}]-g_{\mathcal{F}}(\infty)[e_{i}e_{j}e_{0}^{s-1}]-g_{j}^{-1}[e_{i}]g_{j}[e_{0} ^{s-1}e_{k}]-g_{k} ^{-1}[e_{i}e_{j}e_{0} ^{s-1}]
\end{align*} 

Computing the residues in the differential equation for $dg_{\mathcal{F}}[e_{i}e_{j}e_{0} ^{s} ]$ we obtain that 
\begin{eqnarray*}
g_{\mathcal{F}}(\infty)[e_{i}e_{j}e_{0} ^{s-1}]=g_{\mathcal{F}}(\infty)[e_{j}e_{0} ^{s}]-g_{i}[e_{j}e_{0} ^{s}]=(-1)^{s+1}g_{i}[e_{0} ^{s}e_{j}].
\end{eqnarray*}
 
Using the expression  $g_{\mathcal{F}}[e_{j}e_{0} ^{s-1}e_{k}]=(-1)^{s}sg_{k}[e_{0}^se_{j}]$ that we found in  \textsection \ref{Comp of 1}, we obtain  that $g_{i}[e_{j}e_{0} ^{s-1}e_{k}]=$
\begin{align*} 
g_{k}[e_{i} e_{j}e_{0} ^{s-1}] +(-1)^{s}(sg_{k}[e_{0} ^s e_{j}]+g_{i}[e_{0} ^s e_{j}]+g_{k}[e_{i}] g_{k}[e_{0} ^{s-1}e_{j}] )+g_{j}[e_{i}]g_{j}[e_{0} ^{s-1}e_{k}].
\end{align*}
This proves the following proposition. 
\begin{proposition}
Assume that $i,j,$ and $k$ are pairwise dinstinct and that $s>1.$ Then $g_{i}[e_{j}e_{0} ^{s-1}e_{k}]=$
\begin{align*}
& &g_{k}[e_{i} e_{j}e_{0} ^{s-1}]+\frac{s\mathcal{X}(s+1;j-k)}{1-\zeta^{\underline{j} -\underline{k}}}+\frac{\mathcal{X}(s+1;j-i)}{1-\zeta^{\underline{j} -\underline{i}}}-\frac{\mathcal{X}(1;i-k)}{1-\zeta^{\underline{i}-\underline{k}}} \frac{\mathcal{X}(s;j-k)}{1-\zeta^{\underline{j}-\underline{k}}}\\
& &+(-1)^{s-1}\frac{\mathcal{X}(1;i-j)\mathcal{X}(s;k-j)}{(1-\zeta^{\underline{i}-\underline{j} })(1-\zeta^{\underline{k}-\underline{j}})},
\end{align*}
where $g_{k}[e_{i} e_{j}e_{0} ^{s-1}]$ is given by Proposition \ref{comp of gijk0}.
\end{proposition}

Using the expression for $g_{\mathcal{F}}[e_{l} e_{0} ^{a} e_{m} e_{0} ^{b}]$ above we see that $dg_{\mathcal{F}}[e_{i}e_{j}e_{0} ^{s-1}e_{k}]=$
\begin{align*}
& &d(S(s,1,1;k,j,i;d_{2},d_{3}))+(-1)^{s}S(s,(1),1;j,k,i;d_{3})+g_{j}[e_{0} ^{s-1}e_{k}]S(1,1;i,j)\\
& &+g_{k} ^{-1}[e_{j}e_{0} ^{s-1}]S(1,1;i,k)+(-1) ^s\sum_{0 \leq r \leq s-1} S(1+r,s-r,1;j,i,k;d_{2})\\
& &+(-1) ^{s+1}S(1,(s),1;i,j,k)     )+\sum _{1 \leq i \leq M}\alpha_{i} \omega_{i},
\end{align*}
for some $\alpha_{i} \in \mathbb{Q}_{p}[\zeta].$ By the same arguments as above we see that the hypotheses of Proposition \ref{fun prop} are satisfied, and then this proposition implies the following.  
\begin{proposition}
Suppose that $s\geq1,$ then we have $g_{\mathcal{F}}[e_{i}e_{j}e_{0} ^
{s-1}e_{k}]=$
\begin{align*}
& &S(s,1,1;k,j,i;d_{2},d_{3}))+(-1)^{s}S(s,(1),(1);j,k,i;d_{3})+g_{j}[e_{0} ^{s-1}e_{k}]S(1,(1);i,j)\\
& &+g_{k} ^{-1}[e_{j}e_{0} ^{s-1}]S(1,(1);i,k)+(-1) ^s\sum_{0 \leq r \leq s-1} S(1+r,s-r,(1);j,i,k;d_{2})\\
& &+(-1) ^{s+1}S(1,(s),(1);i,j,k).     
\end{align*}
\end{proposition}

  \subsubsection{Computation of $g_{i}[e_{j}e_{0} ^{s-1}e_{k}e_{0 } ^{t-1} ]$}
  Assume that $s,t >1.$ Then the differential equation gives that  $dg_{\mathcal{F}}[e_{i}e_{j}e_{0} ^{s-1}e_{k}e_{0}^{t-1}]=$ 
  \begin{align*}
& &g_{\mathcal{F}}[e_{j}e_{0} ^{s-1}e_{k}e_{0} ^{t-1}]\mathcal{F}^*\omega_{i}-pg_{\mathcal{F}}[e_{i}e_{j}e_{0} ^{s-1}e_{k}e_{0} ^{t-2}]\omega_{0} -p g_{i}[e_{j}e_{0} ^{s-1}e_{k}e_{0} ^{t-1}]\omega_{\underline{i}}\\
  & &-p g_{j} ^{-1}[e_{i}] g_{j}[e_{0} ^{s-1}e_{k}e_{0}^{t-1}]\omega_{\underline{j}}-pg_{\mathcal{F}}[e_{i}]g_{j}[e_{0} ^{s-1}e_{k}e_{0}^{t-1}]\omega_{\underline{j}}.
  \end{align*}

  We will use Proposition \ref{fun prop} and do induction on $t,$ starting with the formulas we found above for $g_{\mathcal{F}}[e_{i}e_{j}e_{0} ^{s-1}e_{k}]$ and $g_{\mathcal{F}}[e_{a}e_{0} ^{l} e_{b} e_{0} ^m ].$  We find that  
 $g_{\mathcal{F}}[e_{i} e_{j}e_{0} ^{s-1}e_{k} e_{0} ^{t-1}]=$ 
 \begin{align*}
 & &\frac{(-1)^{t-1}}{(s-1)!}\sum_{0 \leq r,q \atop {r+q \leq t-1}} (r+1)_{s-1}S(s+r,1-q,t-(q+r);k,j,i;d_{2},d_{3})\\
 & &+\sum_{0 \leq r \leq t-1}(-1)^{r}g_{j}[e_{0} ^{s-1}e_{k}e_{0} ^{t-1-r}]S(1,(1+r);i,j)\\
 & &+ (-1)^{t-1}g_{k} ^{-1}[e_{j} e_{0} ^{s-1}]S(1,(t);i,k)+(-1)^{s+t}S(1, (s), (t);i,j,k)\\
 & &+(-1)^{s+t+1}\sum_{0 \leq r \leq s-1}S(1+r,s-r,(t);j,i,k;d_{2})\\
 & &+ (-1)^{s+t+1}\sum_{1 \leq r \leq t} S(s, (t+1-r),(r);j,k,i;d_{3}),
 \end{align*}
  for $s\geq 1$ and $t \geq 1.$ 
  
  We also obtain that $-\zeta^{-\underline{i}}g_{i}[e_{j}e_{0} ^{s-1}e_{k}e_{0} ^{t-1}]- \zeta^{-\underline{j}}g_{j} ^{-1}[e_{i}]g_{j}[e_{0} ^{s-1}e_{k}e_{0} ^{t-1}]=$
 \begin{align*}
 & &\zeta^{-\underline{j}}\sum_{0 \leq r \leq t-2}(-1)^{r+1}g_{j}[e_{0} ^{s-1}e_{k}e_{0} ^{t-2-r}]\mathcal{X}^{(1+r)}(1;j-i)\\
 & &+ \zeta^{-\underline{k}}(-1)^{t-1}g_{k} ^{-1}[e_{j} e_{0} ^{s-1}]\mathcal{X}^{ (t-1)}(1;k-i)+\zeta^{-\underline{k}}(-1)^{s+t}\mathcal{X}^{ ( t-1)}(1, (s);j-i,k-j)\\
 & &+\zeta^{-\underline{k}}(-1)^{s+t+1}\sum_{0 \leq r \leq s-1}\mathcal{X}^{ (t-1)}(1+r,s-r;i-j,k-i;d_{2})\\
 & &+\zeta^{-\underline{i}} (-1)^{s+t+1}\sum_{1 \leq r \leq t-1} \mathcal{X}^{ (r)}(s, (t-r );k-j,i-k;n_{2})\\
 & &+\zeta^{-\underline{i}}(-1)^{s+t+1}\mathcal{X}(s, (t);k-j,i-k;n_{2})+\zeta^{-\underline{j}}g_{j} [e_{0} ^{s-1}e_{k}e_{0} ^{t-1}]\mathcal{X}(1;j-i).
\end{align*}   
Using the formula for the polylogarithmic part above, we obtain the following.
\begin{theorem}\label{Theorem main}
Assume that $s,t \geq 2. $ Then  $g_{i}[e_{j}e_{0}^{s-1}e_{k}e_{0} ^{t-1}]=$ 
\begin{eqnarray*}
& &(-1)^{s-1}\zeta^{\underline{i} -\underline{j}}{s+t-2 \choose s-1}\frac{\mathcal{X}(1;i-j)}{1-\zeta^{\underline{i}-\underline{j}}}\frac{\mathcal{X}(s+t-1;k-j)}{1-\zeta^{\underline{k}-\underline{j}}}\\
& &-\zeta^{\underline{i}-\underline{j}}\sum_{0 \leq r \leq t-2}(-1)^{r+s}{s+t-r-3 \choose s-1}\frac{\mathcal{X}(s+t-r-2;k-j)}{1-\zeta^{\underline{k}-\underline{j}}}  \mathcal{X}^{(1+r) }(1;j-i)\\
& &- \zeta^{\underline{i}-\underline{k}}((-1)^{t}\frac{\mathcal{X}(s;j-k)}{1-\zeta^{\underline{j}-\underline{k}}}\mathcal{X}^{( t-1) }(1;k-i)+(-1)^{s+t}\mathcal{X}^{ ( t-1)}(1,  (s);j-i,k-j)\\
& &+(-1)^{s+t+1}\sum_{0 \leq r \leq s-1}\mathcal{X}^{ ( t-1)}(1+r,s-r;i-j,k-i;d_{2}))\\
& &+ (-1)^{s+t}(\sum_{1 \leq r \leq t-1} \mathcal{X}^{ ( r)}(s, (t-r) ;k-j,i-k;n_{2})
+\mathcal{X}(s, (t);k-j,i-k;n_{2}))\\
& &+(-1)^s{s+t-2 \choose s-1}\zeta^{\underline{i}-\underline{j}}\frac{\mathcal{X}(s+t-1;k-j)}{1-\zeta^{\underline{k}-\underline{j}}}\mathcal{X}(1;j-i),
\end{eqnarray*}
for distinct $i,j,k.$
\end{theorem}

\end{document}